\documentclass[11pt]{amsart}

\usepackage{amsmath,amssymb,amsthm}
\usepackage{mathtools}
\usepackage{graphicx}
\usepackage{geometry}
\usepackage{enumitem}
\usepackage{tikz}
\usetikzlibrary{patterns}
\usepackage{tikz-cd}
\usepackage{hyperref}
\usepackage{bookmark}
\usepackage{cleveref}
\usepackage[all]{xy}

\theoremstyle{plain}
\newtheorem{thm}{Theorem}[section]
\newtheorem{lem}[thm]{Lemma}
\newtheorem{prop}[thm]{Proposition}
\newtheorem{cor}[thm]{Corollary}
\theoremstyle{definition}
\newtheorem{defi}[thm]{Definition}
\newtheorem{ex}[thm]{Example}
\theoremstyle{remark}
\newtheorem{rem}[thm]{Remark}

\newcommand{\Ac}{\mathcal{A}}
\newcommand{\Cc}{\mathcal{C}}
\newcommand{\Hc}{\mathcal{H}}
\newcommand{\Sc}{\mathcal{S}}
\newcommand{\ebf}{\mathbf{e}}
\providecommand{\Fb}{\mathbb{F}}
\providecommand{\Qb}{\mathbb{Q}}
\providecommand{\Zb}{\mathbb{Z}}
\DeclareMathOperator{\rk}{rank}
\DeclareMathOperator{\crk}{corank}

\newcommand{\abs}[1]{\left\lvert #1\right\rvert}
\DeclareMathOperator{\im}{\mathrm{Im}} 
\DeclareMathOperator{\kerne}{\mathrm{Ker}} 

\title{A categorification for the characteristic polynomial of matroids}
\author{Takuya Saito}
\address{Institute for Chemical Reaction Design and Discovery, Hokkaido University, Kita 21 Nishi 10, Kita-ku, Sapporo, 001-0021 Japan}
\email{saito@icredd.hokudai.ac.jp}

\author{So Yamagata}
\address{Department of Applied Mathematics, Faculty of Science, Fukuoka University, Jonan-ku Nanakuma 8-19-1, Fukuoka, 814-0180 Japan}
\email{so.yamagata@fukuoka-u.ac.jp}

\date{}

\begin{document}

\begin{abstract}
  In the present paper, we provide cohomology groups for matroids, as a categorification of the characteristic polynomial of matroids. 
       The construction depends on the ``quasi-representation'' of a matroid. 
       For a certain choice of the quasi-representation, we show that our cohomology theory gives a generalization of the chromatic cohomology introduced by L. Helme--Guizon and Y. Rong, and also the characteristic cohomology introduced by Z. Dancso and A. Licata.
\end{abstract}

\maketitle

\section{Introduction}
    Khovanov \cite{Kho} introduced a bigraded cohomology theory of links, whose graded Euler characteristic is the Jones polynomial. 
    In analogous fashion, several constructions of Khovanov-type cohomology theories have been provided beyond knot theory. 
    In \cite{HGR}, the chromatic polynomial of graphs is categorified, and resulting cohomology is known as the chromatic cohomology. 
    In analogy with Khovanov cohomology, its graded Euler characteristic yields the chromatic polynomial. 
    The deletion-contraction operation, a fundamental concept in graph theory, appears as long exact sequences of associated cohomology groups in the categorified context. 
    The properties of this cohomology have been studied \cite{HGPR,LS,PPS,SS}. 
    The second author \cite{Y} explored a splitting property of chromatic cohomology for a specific class of graphs and applied it to describe the cohomology of complete graphs. 
    Tutte cohomology, which categorifies the Tutte polynomial of graphs, was introduced in \cite{JHR}.
    The categorification of the chromatic polynomial of fat graphs is introduced in \cite{LM}. 
    In \cite{SY}, the categorification of Stanley's chromatic symmetric function is introduced.

    In \cite{DL}, three classes of polynomials related to hyperplane arrangements--the characteristic polynomial, the Tutte polynomial, and the Poincar\'{e} polynomial--were categorified. 
    Consider $\Ac = \{ H_1, \dots, H_n \}$, an arrangement of hyperplanes in a $k$-dimensional vector space $V$ over $\Bbbk$. 
    To a subset $S \subset [n] \coloneqq \{ 1, \dots, n \}$, they associated three vector spaces $H_S = \bigcap_{e \in S} H_e$, $V_S = \mathrm{span} \{ \nu_e \}_{e \in S}$ where $\nu_e \in V^\ast$ is the normal vector of $H_e$ and $V^\ast$ is the dual space of $V$, and $W_S = \{ w = (w_1, \dots, w_n) \in W \mid w_r = 0 \text{ for } r \notin S \}$, where $W = \{ (w_1, \dots, w_n) \in \Bbbk^n \mid \sum_e w_e \nu_e = 0 \}$. 
    Then, the chain groups were defined by taking exteriors $\bigwedge^\bullet H_S$, $\bigwedge^\bullet V_S$ and $\bigwedge^\bullet W_S$. 
    The differentials were defined as maps induced by three types of maps: for $s \in S$ and $r \notin S$, natural inclusions $H_S \hookrightarrow H_{S \setminus s}$, $V_S \hookrightarrow V_{S \cup r}$, $W_S \hookrightarrow W_{S \cup r}$, and orthogonal projection maps $H_S \twoheadrightarrow H_{S \cup r}$, $V_S \twoheadrightarrow V_{S \setminus s}$, $W_S \twoheadrightarrow W_{S \setminus s}$.

    The central concept treated in the present paper is that of a matroid. 
    A matroid is a structure that reflects the notion of abstract dependency, including cycles in graphs and linear dependency. 
    Therefore, we can obtain matroids from both graphs and arrangements of hyperplanes (see Example \ref{ExamOfMat}). 
    In the present paper, we provide cohomology groups for matroids, which unify and extend the chromatic cohomology of graphs and the characteristic cohomology of hyperplane arrangements.
    Our constructions are motivated by work on the characteristic cohomology of hyperplane arrangements \cite{DL}. 
    In constructing the cohomology groups, we introduce the notion of a ``quasi-representation'' of a matroid to a free $R$-module, where $R$ is a principal ideal domain with a unit. 
    A representation of a matroid can be described as a hyperplane arrangement whose associated matroid is the given one.
    However, infinitely many matroids admit no such representation.
    Unlike representations of matroids, a quasi-representation always exists for every matroid, even if it is non-representable. 
    Thus, our construction can be applied to all matroids.

    Note that our cohomology groups are defined for a matroid by means of the quasi--representation. 
    We show that the deletion-contraction of matroids induces a long exact sequence of the cohomologies of the associated matroids. 
    Unlike in the case of chromatic and characteristic (co)homologies, we need to separate the cases in which the deleted or contracted element is a coloop. 
    Theorem \ref{thm:LongExactSq} gives the long exact sequence for the former case, and Theorem \ref{thm:coloop} gives the one for the latter case. 
    We study how these cohomology groups behave under basic matroid operations, such as adding loops, parallel elements, taking direct sums, and the relaxation operation, which provides a way to construct a quasi-representation of a non-representable matroid.
    In particular, we will show that the cohomology of relaxations can be calculated using short exact sequences in Theorem \ref{thm:relax}.
    Some sample computations of the cohomology groups are also included. Finally, we consider the relation to other Khovanov-type cohomology theories.

    Let $H_{arr}^\bullet(\Ac,\Bbbk)$ be the characteristic cohomology of an essential arrangement $\Ac$ introduced in \cite{DL} with coefficients in $\Bbbk$, and $H^\bullet (M(\Ac), \rho)$ be the characteristic cohomology of the matroid $M(\Ac)$ associated with $\Ac$. 
    Then, we show the following theorem.
    
    \begin{thm}
    Let $H_{arr}^\bullet(\Ac,\Bbbk)$ be the characteristic cohomology of $\Ac$ over $\Bbbk$, and let $H^\bullet(M(\Ac),\rho)$ be the characteristic cohomology of the associated matroid with the above $\Bbbk$-quasi-representation. 
    Then
    \begin{equation*}
        H_{arr}^\bullet(\Ac,\Bbbk) \simeq H^\bullet(M(\Ac),\rho).
    \end{equation*}
    \end{thm}
    Let $H^\bullet(G)$ be the chromatic cohomology of a graph $G$ introduced in \cite{HGR}. 
    We show the existence of a homomorphism from $H^\bullet(G)$ to $H^\bullet(M(G))$ where $M(G)$ is the graphic matroid associated to $G$ (see Theorem \ref{thm: Homomorph}). 
    Moreover, we show that the following theorem applies if $G$ is a connected graph.
    \begin{thm}
        [Theorem \ref{cor:long_ex}] Let $G$ be a connected graph. 
        Then there is a long exact sequence
        \begin{equation*}
            \cdots \to H^{i,j-1}(M(G)) \to H^{i,j}(G) \to H^{i,j}(M(G)) \to H^{i+1,j-1}(M(G)) \to \cdots.
        \end{equation*}
    \end{thm}

    The paper is organized as follows. Section \ref{sec:matroid} reviews basic definitions and notions of matroid theory with some examples. 
    Section \ref{sec:construction} provides a construction of the characteristic cohomology groups beginning with defining a quasi-representation to a free $R$-module of rank $r$. 
    Section \ref{sec:exactseq} proves long exact sequences of the cohomology groups, from which the deletion-contraction formula for characteristic polynomials follows.
    Compared to other Khovanov-type (co)homologies associated with deletion-contraction phenomena \cite{DL,HGR}, our cohomology requires a distinction according to whether the deleted element is a coloop or not.
    Section \ref{sec:properties} investigates the basic properties of our cohomology groups. 
    The relation to the relaxation operation of matroids is also studied. 
    Section \ref{sec:computation} provides sample computations of our cohomology groups. 
    In particular, we compute the cohomology of uniform matroids with small numbers. 
    Finally, Section \ref{sec:relation} compares our cohomology group with two others: characteristic cohomology for hyperplane arrangements and chromatic cohomology for graphs.

\section*{Acknowledgement}
This work was supported by JSPS KAKENHI Grant Numbers JP23KJ0031, JP24K-16926 and JPJSBP120256504.

\section{Matroid and characteristic polynomial}\label{sec:matroid}
    In this section we review the basic notions in matroid theory. 
    For more details of the general matroid theory, see \cite{O} (see also \cite{BO} for the characteristic polynomials). 
    In the rest of the present paper, we denote the singleton as $e$ rather than $\{ e \}$ except in cases where this might create confusion.
    \begin{defi}
        A \textbf{matroid} $M$ is a pair $(E,\rk_M)$ of a finite set $E$ called the \textbf{ground set}, and a non-negative integer valued function $\rk_M: 2^{E} \to \Zb_{\geq 0}$ called the \textbf{rank function} of $M$, with the following properties:
        \begin{enumerate}
            \item $0 \leq \rk_M(S) \leq \abs{S}$;
            \item $S_1 \subset S_2 \subset E$ implies $\rk_M(S_1) \leq \rk_M(S_2)$;
            \item $\rk_M(S_1) + \rk_M(S_2) \geq \rk_M(S_1 \cap S_2) + \rk_M(S_1 \cup S_2)$ for all $S_1, S_2 \subset E$.
        \end{enumerate}
        The rank of a matroid $M$ is defined by the rank of its ground set, i.e., $\rk(M) = \rk_M(E)$.
    \end{defi}

    The \textbf{corank} of a subset $S$ in a matroid $M$ is defined by $\crk_M(S) = \rk(M) - \rk_M(S)$. 
    An element $e$ of the matroid $M$ is called \textbf{loop} if $\rk_M(e) = 0$ and \textbf{coloop} (sometimes called isthmus) if $\crk_M(E\setminus e) = 1$. 
    Moreover, if $e_1$ and $e_2$ are non-loop elements in the matroid $M$ such that $\rk_M(\{ e_1,e_2 \}) = 1$, then $e_1$ and $e_2$ are called \textbf{parallel}.

    Two matroids $M = (E,\rk_M)$ and $M' = (E',\rk_{M'})$ are \textbf{isomorphic} if there exists a bijection $\phi: E \to E'$ such that $\rk_M(S) = \rk_{M'}(\phi(S))$ for all $S \in 2^{E}$.

    \begin{ex}\label{ExamOfMat}\mbox{}
        \begin{itemize}
            \item The \textbf{uniform matroid} $U_{k,n}$ is a matroid on the ground set $\{ 1, \dots, n \}$. 
            The rank function of $U_{k,n}$ is defined by $\rk_{U_{k,n}}(S) = \min \{ k,\abs{S} \}$ for all subsets $S$ of the ground set.
            \item Let $G$ be a graph with a vertex set $V(G)$ and an edge set $E(G)$. 
            As edges we allow loops and multiple edges. 
            The \textbf{graphic matroid} $M(G)$ of $G$ is defined on the edge set $E(G)$ of $G$ and its rank function is defined by $\rk_{M(G)}(S) = \abs{V} - k([G;S])$ for all subsets $S$ of $E(G)$, where $[G; S]$, called the \textbf{spanning graph}, denotes the subgraph of $G$ with the same vertex set $V(G)$ and an edge set $S \subset E(G)$, and $k([G;S])$ denotes the number of connected components of the spanning graph $[G;S]$.
            \item Let $V$ be a vector space of dimension $d$. 
            A \textbf{hyperplane} is a linear subspace of dimension $d - 1$. 
            An \textbf{arrangement of hyperplanes} $\Ac$ is a finite collection of hyperplanes in $V$. 
            Then $\Ac$ is the ground set of a matroid, denoted $M(\Ac)$, whose rank function is defined by $\rk_{M(\Ac)}(S) = d - \dim \bigcap_{H \in S}H$ for $H \in S \subset \Ac$. 
            For the general theory of hyperplane arrangements, see also \textup{\cite{OT}}.
        \end{itemize}
    \end{ex}

    There are some fundamental operations on matroids. 
    Let $M = (E,\rk_M)$ and $M' = (E',\rk_{M'})$ be matroids on disjoint ground sets and let $e \in E$. 
    The \textbf{deletion} $M \setminus e$ and the \textbf{contraction} $M/e$ are matroids defined on $E \setminus e$ with the rank functions
    \begin{align*}
        \rk_{M \setminus e}(S) & = \rk_M(S);                   \\
        \rk_{M/e}(S)         & = \rk_M(S \cup e) - \rk_M(e)
    \end{align*}
    for all $S \in 2^{E \setminus \{ e \}}$. 
    The \textbf{direct sum} $M \oplus M'$ is the matroid defined on the ground set $E \cup E'$ with the rank function
    \begin{equation*}
        \rk_{M \oplus M'}(S) = \rk_M(S \cap  E) + \rk_{M'}(S \cap E')
    \end{equation*}
    for all $S \in 2^{E \cup E'}$. 
    A subset $S_0$ of $E$ is a \textbf{circuit-hyperplane} if 
    \begin{gather*}
        \rk_M(S_0) = \abs{S_0} - 1 = \rk_M(S_0 \setminus \{ e_1 \})\\
        \rk_M(S_0) + 1 = \rk(M) = \rk_M(S_0 \cup \{ e_2 \})
    \end{gather*}
    for all $e_1 \in S_0,\ e_2 \in E \setminus S_0$. 
    Then the following function $\rk_{M^-}$ on $E$ defined by
    \begin{equation*}
        \rk_{M^-}(S) = \left\{
        \begin{aligned}
            \rk_M(S) &  &  & (S \neq S_0), \\
            \rk_M(E) &  &  & (S = S_0),
        \end{aligned}
        \right.
    \end{equation*}
    satisfies the axioms of a matroid. 
    We say $M^-$ is obtained from $M$ by \textbf{relaxing} $S_0$, and $M^-$ is called the {relaxation} of $M$.

    One important example of relaxations is the non-Pappus matroid, which is obtained from the Pappus matroid by relaxing one collinearity.
    \begin{ex}\label{ex:Pappus}
        \begin{figure}
            \centering
            \begin{tabular}{cc}
                \begin{minipage}{0.4\linewidth}
                \centering 
                \begin{tikzpicture}[scale=1]
                \node[inner sep=0.2em, fill=black!100, circle] (1) at (-2,1) {};
                \node[inner sep=0.2em, fill=black!100, circle] (2) at (0,1) {};
                \node[inner sep=0.2em, fill=black!100, circle] (3) at (2,1) {};
                \node[inner sep=0.2em, fill=black!100, circle] (4) at (-1,0) {};
                \node[inner sep=0.2em, fill=black!100, circle] (5) at (0,0) {};
                \node[inner sep=0.2em, fill=black!100, circle] (6) at (1,0) {};
                \node[inner sep=0.2em, fill=black!100, circle] (7) at (-2,-1) {};
                \node[inner sep=0.2em, fill=black!100, circle] (8) at (0,-1) {};
                \node[inner sep=0.2em, fill=black!100, circle] (9) at (2,-1) {}; 
                \node (11) at (-2,1.3) {$1$};
                \node (22) at (0,1.3) {$2$};
                \node (33) at (2,1.3) {$3$};
                \node (44) at (-1,0.3) {$4$};
                \node (55) at (0,0.3) {$5$};
                \node (66) at (1,0.3) {$6$};
                \node (77) at (-2,-1.3) {$7$};
                \node (88) at (0,-1.3) {$8$};
                \node (99) at (2,-1.3) {$9$}; 
                \begin{scope}[line width = 1pt]
                \draw (1)--(3); 
                \draw (4)--(6);
                \draw (7)--(9);
                \draw (1)--(8);
                \draw (1)--(9);
                \draw (2)--(7);
                \draw (2)--(9);
                \draw (3)--(7);
                \draw (3)--(8);
                \end{scope}
                \end{tikzpicture}\caption{Pappus matroid}\label{fig:Pappus}
                \end{minipage} & 
                \begin{minipage}{0.4\linewidth}
                \centering \begin{tikzpicture}[scale=1]
                \node[inner sep=0.2em, fill=black!100, circle] (1) at (-2,1) {};
                \node[inner sep=0.2em, fill=black!100, circle] (2) at (0,1) {};
                \node[inner sep=0.2em, fill=black!100, circle] (3) at (2,1) {};
                \node[inner sep=0.2em, fill=black!100, circle] (4) at (-1,0) {};
                \node[inner sep=0.2em, fill=black!100, circle] (5) at (0,0) {};
                \node[inner sep=0.2em, fill=black!100, circle] (6) at (1,0) {};
                \node[inner sep=0.2em, fill=black!100, circle] (7) at (-2,-1) {};
                \node[inner sep=0.2em, fill=black!100, circle] (8) at (0,-1) {};
                \node[inner sep=0.2em, fill=black!100, circle] (9) at (2,-1) {}; 
                \node (11) at (-2,1.3) {$1$};
                \node (22) at (0,1.3) {$2$};
                \node (33) at (2,1.3) {$3$};
                \node (44) at (-1,0.3) {$4$};
                \node (55) at (0,0.3) {$5$};
                \node (66) at (1,0.3) {$6$};
                \node (77) at (-2,-1.3) {$7$};
                \node (88) at (0,-1.3) {$8$};
                \node (99) at (2,-1.3) {$9$}; 
                \begin{scope}[line width = 1pt]
                \draw (1)--(3);
                \draw (7)--(9);
                \draw (1)--(8);
                \draw (1)--(9);
                \draw (2)--(7);
                \draw (2)--(9);
                \draw (3)--(7);
                \draw (3)--(8);
                \end{scope}
                \end{tikzpicture} \caption{non-Pappus matroid} \label{fig:nonPappus}\end{minipage}
            \end{tabular}
        \end{figure}
        Let $E = \{ 1, \dots, 9 \}$ and let
        \begin{equation*}
            \Sc = \{ 123, 148, 159, 247, 269, 357, 368, 456, 789 \}.
        \end{equation*}
        Define $\rk:2^{E} \to \Zb_{\geq 0}$ as follows:
        \begin{equation*}
            \rk(S) = \left\{
            \begin{aligned}
                 & \abs{S} &  & \abs{S} \leq 2,    \\
                 & 2       &  & S\in \Sc, \\
                 & 3       &  & \mbox{others}.
            \end{aligned}\right.
        \end{equation*}
        Then $M = (E, \rk)$ satisfies the axioms of a matroid and is called the Pappus matroid. 
        The set of circuit-hyperplane in the Pappus matroid is $\Sc$ and it corresponds to the lines in Figure \ref{fig:Pappus}. 
        Relaxation of the Pappus matroid is independent of the choice of relaxing the circuit-hyperplane up to isomorphism. 
        The non-Pappus matroid is the matroid obtained from $M$ by relaxing $\{ 4, 5, 6 \}$, and belongs to a class of matroids, called non-representable, consisting of matroids which cannot be obtained from any arrangements of hyperplanes over commutative fields. 
        Note that we can still represent the non-Pappus matroid over a skew field (see Section 6.1 in \cite{O} and also Section 3.5 in \cite{PVZ} for the representation matrix).
    \end{ex}

    The following proposition follows immediately from the definitions.

    \begin{prop}\label{prop:corank} 
    Let $M$ and $M'$ be matroids on disjoint ground sets $E$ and $E'$, and let $e \in E$, $S \subset E \setminus e$, $S' \subset E'$.
        \begin{itemize}
            \item $\crk_{M \setminus e}(S) = \crk_M(S) - \crk_M(E \setminus e)$;
            \item $\crk_{M/e}(S) = \crk_M(S \cup e)$;
            \item $\crk_{M \oplus M'}(S \cup S') = \crk_M(S) + \crk_{M'}(S')$.
        \end{itemize}
    \end{prop}

    \begin{defi}
        Let $M$ be a matroid $M$ on ground set $E$. 
        The \textbf{characteristic polynomial} $\chi(M;\lambda)$ of $M$ is defined by
        \begin{equation*}
            \chi(M;\lambda) = \sum_{S \in 2^E}(-1)^{\abs{S}} \lambda^{\crk(S)}.
        \end{equation*}
    \end{defi}

    We close the subsection with well-known properties of the characteristic polynomials.

    \begin{prop}
        Let $G=(V(G),E(G))$ be a graph and $\lambda \in \mathbb{N}$. 
        A $\lambda$-coloring of $G$ is a set map $f:  \to [\lambda]$ satisfying the condition that for every edge, its distinct end vertices $v, w$ have distinct values $f(v), f(w)$. 
        Let $P(G;\lambda)$ be the number of $\lambda$-colorings of $G$, which actually turns out to be a polynomial called the \textbf{chromatic polynomial} of $G$. 
        Then
        \begin{equation*}
            P(G;\lambda) = \lambda^{k(G)} \chi(M(G);\lambda)
        \end{equation*}
        where $k(G)$ is the number of connected components of $G$.
    \end{prop}

    \begin{prop}
        Let $\Ac$ be an arrangement of hyperplanes. 
        Define its \textbf{characteristic polynomial} by
        \begin{equation*}
            \chi(\Ac;\lambda) = \sum_{S \subset \Ac}(-1)^{\abs{S}} \lambda^{\dim H_S}
        \end{equation*}
        where $H_S = \bigcap_{H \in S} H$. 
        Then,
        \begin{equation*} \chi(\Ac;\lambda) = \lambda^{\dim \bigcap_{H \in \Ac} H} \chi(M(\Ac);\lambda).
        \end{equation*}
    \end{prop}
    \begin{prop}\label{prop:CharPoly} 
        Let $M$ be a matroid on $E$. 
        Then, the following hold.
        \begin{enumerate}
            \item If $E = \emptyset$, then $\chi(M;\lambda) = 1$.
            \item If $e \in E$ is not a coloop, then $\chi(M;\lambda) = \chi(M \setminus e;\lambda) - \chi(M/e;\lambda)$.
            \item If $e \in E$ is a coloop, then $\chi(M;\lambda) = (\lambda - 1) \chi(M/e;\lambda)$.
            \item If $e \in E$ is a loop, then $\chi(M;\lambda) = 0$.
            \item If $e, e' \in E$ are parallel, then $\chi(M;\lambda) = \chi(M \setminus e;\lambda)$.
            \item For matroids $M_1, M_2$, $\chi(M_1 \oplus M_2;\lambda) = \chi(M_1;\lambda) \chi(M_2;\lambda)$.
            \item Let $M^{-}$ be obtained from $M$ by relaxing a circuit-hyperplane of $M$. 
            Then, $\chi(M^-;\lambda) = \chi(M;\lambda) + (-1)^{\rk(M) - 1}(\lambda - 1)$.
        \end{enumerate}
    \end{prop}
    \section{Construction of the characteristic cohomology for matroids}\label{sec:construction} 
    In this section, we construct the characteristic cohomology for matroids. 
    The construction of the chain groups and differentials are inspired by the work \cite{DL}.

    For a matroid $M$ of rank $r$ on the ground set $E = \{ 1, \dots, n \}$ with the order $1 < \dots < n$, we write the ground set as $E = \{ 1 < \dots < n \}$ for simplicity.

    Suppose $R$ is a principal ideal domain with a unit. 
    For a finitely generated $R$-module $N$, define $\rk_R N$ as the dimension of $N \otimes_R \Bbbk$ where $\Bbbk$ is the quotient field of $R$. 
    It is well-known that there exists an isomorphism $N \cong R^r \oplus R/(c_{r+1}) \oplus \dots \oplus R/(c_{r+t})$ for some $r$ and $c_{r+1}, \dots, c_{r+t} \in R$, where $c_{r+l}$ divides $c_{r+l+1}$ for $l = 1, \cdots, t - 1$. 
    We take the generators by $\bar{\ebf}_1, \dots, \bar{\ebf}_r, \bar{\ebf}_{r+1}, \dots, \bar{\ebf}_{r+t}$ where $\bar{\ebf}_1, \dots, \bar{\ebf}_r$ generate the free part of $N$ and $\bar{\ebf}_{r+l}$ generates the torsion part of $N$ isomorphic to $R/(c_{r+l})$.
    \subsection{Quasi-representation of matroids}
    In this subsection, we introduce some representations of matroids. 
    The construction is inspired by the study \cite{FM}, in which the authors introduce the notion of a matroid over a commutative ring which assigns to every subset of the ground set an $R$-module.

    First, we give a quick view of the representation of matroids.
    \begin{defi}
        Let $\Bbbk$ be a field. 
        A $\Bbbk$-\textbf{representation} of a matroid $M$ is the map $\bar{\rho}$ from a ground set $E$ to a vector space $N$ with the condition $\rk_M(S) = \dim_{\Bbbk} \left< \bar{\rho}(e) \mid e \in S \right>$ for all subsets $S \subset  E$. 
        A matroid $M$ is \textbf{$\Bbbk$-representable} if there exists a $\Bbbk$-representation of $M$.
    \end{defi}

    The following generalizes the notion of the $\Bbbk$-representation.
    \begin{defi}
        An $R$-\textbf{quasi-representation}, or simply quasi-representation if we do not need to specify $R$, of matroid $M$ is a map $\rho$ assigning to each subset $S \in 2^E$ a submodule $\rho(S)$ of $N$ such that:
        \begin{itemize}
            \item $\rk_R \rho(S) = \rk_M(S)$, $\rho(\emptyset) = 0$, and $\rho(E) = N$;
            \item if $S_1 \subset S_2 \subset E$, then $\rho(S_1)\subset \rho(S_2)$;
            \item $\rho(S_1) = \rho(S_2)$ whenever $\rk(S_1) = \rk(S_2)$ and $S_1\subset S_2$.
        \end{itemize}
    \end{defi}

    We denote the matroid with quasi-representation by $(M,\rho) = (E,\rk_M, \rho)$.

    \begin{rem}
        A quasi-representation is a functor from the Boolean lattice $2^E$ to the category $R$-mod of finitely generated $R$-modules, which sends morphisms to inclusions and satisfies the appropriate rank conditions.
        Another similar notion of quasi-representation is discussed in \cite{BEST}.
    \end{rem}

    \begin{rem}
        We define the characteristic polynomial of a matroid with a fixed quasi-representation by the characteristic polynomial of the underlying matroid.
    \end{rem}
    Quasi-representations $\rho,\rho'$ of $M$ on ground set $E$ are \textbf{equivalent} if there exists a module isomorphism $\psi$ between $\rho(E)$ and $\rho'(E)$ with the condition that for all subsets $S \subset E$, $\psi(\rho(S)) = \rho'(S)$.
    That is, the assignment $\rho$ defines a functor from the Boolean lattice $2^E$ to the poset of submodules of $N$, and the following diagram commutes.
    \begin{equation*}
        \xymatrix{ & S \ar@{|->}[dl]_{\rho} \ar@{|->}[dr]^{\rho'} & \\ \rho(S) \ar@{|->}[rr]_{\psi} & & \rho'(S) }
    \end{equation*}

    Let us define the operations on matroids with quasi-representations.
    \begin{defi}
        Let $\rho$ and $\rho'$ be quasi-representations of matroids $M$ and $M'$ on disjoint ground sets $E$ and $E'$, respectively, and $e \in E$. 
        Then,
        \begin{itemize}
            \item the deletion $(M \setminus  e,\rho_{\setminus e})$ of $(M,\rho)$ is the matroid of the deletion of $M$ by $e$ with the quasi-representation $\rho_{\setminus e}$ assigning to each subset $S \in 2^{E \setminus e}$ a submodule $\rho(S)$ of $\rho(E \setminus e)$;
            \item the contraction $(M/e,\rho_{/e})$ of $(M,\rho)$ is the matroid of the contraction of $M$ by $e$ with the quasi-representation $\rho_{/e}$ assigning to each subset $S \in 2^{E \setminus e}$ the submodule $\rho(S \cup e)/\rho(e)$ of $\rho(E)/\rho(e)$;
            \item the direct sum $(M \oplus M', \rho \oplus \rho')$ of $(M,\rho)$ and $(M',\rho')$ is the matroid of the direct sum $M \oplus M'$ with the quasi-representation $\rho \oplus \rho'$ assigning to each subset $S\in 2^{E \cup E'}$ a submodule $\rho(S \cap  E) \oplus \rho'(S \cap E')$ of $\rho(E) \oplus \rho'(E')$.
        \end{itemize}
        In addition, we define the relaxation of $\rho$. 
        Let $M^-$ be obtained from $M$ by relaxing a circuit-hyperplane $S_0$ of $M$. 
        Then,
        \begin{itemize}
            \item the relaxation $(M^-,\rho^-)$ is the matroid of $M^-$ with the quasi-representation $\rho^-$ satisfying $\rho^-(S_0) = \rho(E)$ and $\rho^-(S) = \rho(S)$ if $S \neq S_0$.
        \end{itemize}
    \end{defi}

    It is easy to see that these operations are well-defined. Hereinafter, the term matroid will be used not only for a matroid $M$ but also a matroid $(M,\rho)$ with a fixed quasi-representation $\rho$.

    All of the matroids in Example \ref{ExamOfMat} are $\Bbbk$-representable for some $\Bbbk$, but there are infinitely many non-representable matroids which are not $\Bbbk$-representable for any $\Bbbk$ (see \cite{N}). 
    Some examples of non-representable matroids are the non-Pappus matroid in \ref{ex:Pappus}, the non-Desargues matroid and the V\'amos cube (for details see \cite{O}). 
    However, any matroid $M$ has a quasi-representation. Some non-representable matroids are obtained from representable matroids by relaxing circuit-hyperplanes $S_0$--e.g., the Pappus matroid and the non-Pappus matroid. 
    In this case, we can choose a quasi-representation $\rho$ such that $\rho(S_0) = N$, and for any other subset $S \neq S_0$, the value $\rho(S)$ is given by the submodule determined by a fixed quasi-representation of $M$. 
    In more general cases, quasi-representations are chosen as follows. 
    For a subset $S \subset E$, define a free module $\rho(S) = \left< \ebf_1, \dots, \ebf_{\rk_M(S)} \right>_r$.

    A matroid which is $\Bbbk$-representable for all fields $\Bbbk$ is called a
    \textbf{regular} matroid.
    It is well known that every regular matroid admits a representation by a
    totally unimodular matrix over $\Zb$.
    Recall that a matrix is totally unimodular if all of its square subdeterminants
    belong to $\{ -1, 0, 1 \}$.
    Regular matroids have the following uniqueness property.
    \begin{prop}[Proposition 6.6.5 in \cite{O}]
    Let $M$ be an $\Fb_2$-representable matroid and let $\Bbbk$ be a field.
    If $M$ is $\Bbbk$-representable, then all $\Bbbk$-representations of $M$ are
    projectively equivalent. In particular, $M$ is uniquely $\Bbbk$-representable.
    \end{prop}
    One family of regular matroids is the class of graphic matroids.
    Let $G=(V(G),E(G))$ be a simple graph.
    After choosing an orientation of each edge, the $\Bbbk$-representation
    $\bar{\rho}$ of $M(G)$ is defined by
    $\bar{\rho}(e) = v - v' \in \bigoplus_{v \in V(G)} \Bbbk v$
    for an edge $e$ oriented from $v'$ to $v$.
    
    Let $M$ be a regular matroid of rank $r$ on $E$.
    Choose a totally unimodular matrix representing $M$ over $\Qb$, and write $\bar{\rho}: E \to \Zb^r\subset \Qb^r$ for the map assigning to each element of $E$ the corresponding column vector.
    For a principal ideal domain $R$, set $N = \Zb^r \otimes_{\Zb} R \cong R^r$.
    Then we define an $R$-quasi-representation $\rho_R$ of $M$ by $\rho_R(S) = \left< \bar{\rho}(e) \mid e \in S \right>_{\Zb} \otimes R \subset N$ for each $S\subset E$.
    Since the representing matrix is totally unimodular, the rank of each submatrix
    is preserved after tensoring with arbitrary fields.
    Hence we have $\rho_R (E) = N$.
    Therefore $\rho_R$ is an $R$-quasi-representation of $M$.
    We call $\rho_R$ the \textbf{canonical} $R$-quasi-representation.
    If the quasi-representation $\rho$ is canonical, we abbreviate $(M,\rho)$
    to $M$.

    \subsection{The construction of the characteristic cohomology for matroids}
    We construct the chain complex $\Cc = (C^\bullet(M,\rho), d_{(M,\rho)}^\bullet)$ in this subsection. 
    Let $A = \bigoplus_j A_j$ be a graded $R$-module, where $A_j$ denotes the set of homogeneous elements with degree $j$. 
    We call the power series $q \mathchar`-\! \rk A = \sum_j q^j \rk A_j$ the \textbf{graded rank} of $A$.

    For graded $R$-modules $A = \bigoplus_j A_j$ and $B = \bigoplus_j B_j$ where $A_j$ (resp.\ $B_j$) denotes the set of elements of $A$ (resp.\ $B$) of degree $j$, an $R$-module map $\alpha: A \to B$ is called \textbf{graded with degree $d$} if $\alpha(A_j) \subset B_{j+d}$. 
    A chain complex is called a \textbf{graded chain complex} if the chain groups are graded $R$-modules and the differentials are graded.
    \begin{defi}
        The \textbf{graded Euler characteristic} $\chi_{q}(\Cc)$ of a graded chain complex $\Cc$ is defined by
        \begin{equation}
            \chi_{q}(\Cc) = \sum_{0 \leq i \leq n}(-1)^i q \mathchar`-\!\rk C^{i,\bullet}.
        \end{equation}
    \end{defi}

    As a first step, we construct the chain groups. 
    Let $S$ be a subset of $E$. 
    For a matroid with a quasi-representation $(M,\rho)$, we define the module $C^{S,j}(M,\rho)$ as an exterior algebra of the quotient group of $N = \rho(E)$ by the image of the quasi-representation of $S$; i.e., we define it as
    \begin{equation*}
        C^{S,j}(M,\rho) = \bigwedge^j \left( N/\rho(S) \right).
    \end{equation*}
    The $(i,j)$-th chain group is defined by $C^{i,j}(M,\rho) = \bigoplus_{S \in \binom{E}{i}} C^{S,j}(M,\rho)$, where $\binom{E}{i}$ is the set of all subsets of $E$ of cardinality $i$. 
    Taking a sum over $j$, we have the $i$-th chain group as $C^i(M,\rho) = \bigoplus_{j \geq 0}C^{i,j}(M,\rho)$.

    As the next step, let us define differentials. 
    For $S \subset E$ and $e \notin S$ let $d_{(M,\rho)}^{S,e}: \bigwedge^j \left( N / \rho(S) \right) \to \bigwedge^j \left( N / \rho (S \cup e) \right)$ be an algebra map induced by module homomorphism $N/\rho(S) \to N/\rho (S \cup e)$. 
    The $(i,j)$-th differential is defined by
    \begin{equation*}
        d_{(M,\rho)}^{i,j}= \bigoplus_{\substack{S \in \binom{E}{i}\\ e \notin S}}\epsilon^{S,e}d_{(M,\rho)}^{S,e}.
    \end{equation*}
    The value $\epsilon^{S,e}$ is defined by
    \begin{align*}
        \epsilon^{S,e} & = \left\{\begin{aligned}-1&&&\mbox{if the cardinality of }\{e'\in S\mid e'<e\} \mbox{ is odd,} \\ 1&&&\mbox{otherwise.}\end{aligned}\right.
    \end{align*}
    Taking a sum over $j$, we have the $i$-th differential as $d_{(M,\rho)}^i= \bigoplus_{j \geq 0}d_{(M,\rho)}^{i,j}$.

    Note that since $d_{(M,\rho)}^{S \cup e, e'} \circ d_{(M,\rho)}^{S,e} = d_{(M,\rho)}^{S \cup e',e} \circ d_{(M,\rho)}^{S,e'}$ and $\epsilon^{S,e} \epsilon^{S \cup e,e'}= - \epsilon^{S,e'}\epsilon^{S \cup e',e}$, it is clear that $d_{(M,\rho)}^{i+1,j}\circ d_{(M,\rho)}^{i,j} = 0$.
    Thus we have a chain complex $\Cc(M,\rho) = (C^\bullet(M,\rho),$ $d_{(M,\rho)}^\bullet)$.
    \begin{defi}
        We call the cohomology of the chain complex $\Cc(M,\rho)$ the \textbf{characteristic cohomology of the matroid $(M,\rho)$}, and denote it by $H^\bullet(M,\rho)$.
    \end{defi}
    By the above construction, the cohomology for $j = 0$ vanishes unless $M = U_{0,0}$ since $\bigwedge^0 L = R$ for all $R$-module $L$.
    \begin{thm}
        The graded Euler characteristic of the chain complex $\Cc(M,\rho)$ is equal to the characteristic polynomial of the matroid $M$ evaluated at $\lambda = 1 + q$.
    \end{thm}
    Note that the graded Euler characteristic of the chain complex $\Cc(M,\rho)$ is independent from the choice of quasi-representation.

    The definition of the differential relies on the ordering of elements of the matroid and a quasi-representation. 
    Though the cohomology depends on the choice of the quasi-representation, it is determined independently of the ordering of elements in the matroid. This fact is a generalization of Lemma 2.2 in \cite{DL}.
    \begin{lem}
        For any permutation $\sigma \in S_{n}$ and $(M,\rho) = (E, \rk_M, \rho)$ whose ground set is $E = \{ 1 < \dots < n \}$, let $\sigma(M,\rho) = (E, \rk_M, \rho)$ be a matroid with permuted order $\sigma(1)< \dots< \sigma(n)$. 
        Then, we have
        \begin{equation*}
            H^\bullet(M,\rho) \simeq H^\bullet(\sigma(M,\rho)).
        \end{equation*}
    \end{lem}
    \begin{proof}
        The proof is similar to the proof of Lemma 2.2 in \cite{DL} by replacing the arrangement of hyperplanes $\Hc= \{ V; H_1, \dots, H_{n}\}$ and $\mathcal{H_\sigma}= \{ V; H_{\sigma(1)}, \dots, H_{\sigma(n)} \}$ by matroids $(M,\rho) = (E, \rk_M, \rho)$ with $E = \{ 1 < \dots < n \}$ and $\sigma(M,\rho) = (E, \rk_M, \rho)$ with $\sigma(1) < \dots < \sigma(n)$, respectively.
    \end{proof}

    \begin{rem}
        As noted above, the cohomology $H^\bullet(M,\rho)$ depends on the choice of the quasi-representation. 
        However, if the representations $\rho,\rho'$ of $M$ are equivalent, we can see that $H^\bullet(M,\rho) \simeq H^\bullet(M,\rho')$. 
        In this remark, let us consider an example in which $H^\bullet(M,\rho)$ depends on the quasi-representations. 
        Let $E = \{ 1, 2 \}$ be a ground set and consider two quasi-representations $\rho_1$ and $\rho_2$ of the uniform matroid $U_{2,2}$ defined by $\rho_1(1) = \left< \ebf_1 \right>, \rho_1(2) = \left< \ebf_2 \right>$ and $\rho_2(1) = \rho_2(2) = \left< \ebf_1\right>$, where $\ebf_1 = (1,0), \ebf_2 = (0,1)$ are generators of $N = R^2$. 
        Let $(M,\rho_i)$, $i = 1, 2$ be matroids on the same ground set $E$ with two quasi-representations $\rho_1, \rho_2$.
        Let $\rk_M(S) = \abs{S}$. 
        Then, we have chain groups as
        \begin{align*}
            C^{\{ 1 \}}(M,\rho_1) = C^{\{ 1 \}}(M,\rho_2) = C^{\{ 2 \}}(M,\rho_2) & = \bigwedge^\bullet \left( N/\left< e_1 \right> \right)  \\
            C^{\{ 2 \}}(M,\rho_1)                                                     & = \bigwedge^\bullet \left(N/\left< e_2 \right> \right).
        \end{align*}
        When $j = 0$, obviously $C^{*,0}(M,\rho) = R$, so we have $H^{i,0}(M,\rho_1) = H^{i,0}(M,\rho_2)$. 
        Similarly, we have $H^{i,2}(M,\rho_1) = H^{i,2}(M,\rho_2)$. 
        The difference appears when $j = 1$. 
        The chain group $C^{i,1}(M,\rho_1)$ would be as follows.

        \begin{equation*}
            \xymatrix{ & R^2/\left< (1,0) \right> \ar@{->>}[dr]^{d_{(U_{2,2})_{\rho_1}}^{\{ 1 \}, \{ 2 \}}} & \\ 
            R^2 \ar@{->>}[ur]^{d_{(U_{2,2})_{\rho_1}}^{\emptyset, \{ 1 \}}} \ar@{->>}[dr]_{d_{(U_{2,2})_{\rho_1}}^{\emptyset, \{ 2 \}}} & & 0 \\ 
            & R^2/\left< (0,1) \right> \ar@{->>}[ur]_{d_{(U_{2,2})_{\rho_1}}^{\{ 2 \}, \{ 1 \}}} & }
        \end{equation*}

        In this case, we have $H^{0,1}(M,\rho_1) = H^{1,1}(M,\rho_1) = H^{2,1}(M,\rho_1) = 0$. 
        On the other hand, the chain group $C^{i,1}(M,\rho_2)$ would be as follows.
        \begin{equation*}
            \xymatrix{ & R^2/\left< (1,0) \right> \ar@{->>}[dr]^{d_{(U_{2,2})_{\rho_2}}^{\{ 1 \}, \{ 2 \}}} & \\ 
            R^2 \ar@{->>}[ur]^{d_{(U_{2,2})_{\rho_2}}^{\emptyset, \{ 1 \}}} \ar@{->>}[dr]_{d_{(U_{2,2})_{\rho_2}}^{\emptyset, \{ 2 \}}} & & 0 \\ 
            & R^2/\left<(1,0)\right> \ar@{->>}[ur]_{d_{(U_{2,2})_{\rho_2}}^{\{ 2 \}, \{ 1 \}}} & }
        \end{equation*}
        In this case, we have $H^{0,1}(M,\rho_2) = \left< (1,0) \right> \simeq R$, $H^{1,1}(M,\rho_2) \simeq R$, $H^{2,1}(M,\rho_2) = 0$.
    \end{rem}

    \section{Basic properties of the characteristic cohomology for matroids}\label{sec:properties} 
    In this section, we categorify the properties of the characteristic polynomial of matroids in Proposition \ref{prop:CharPoly}. 
    Throughout this section, we assume $M \neq U_{0,0}$.

    \subsection{Deletion-contraction property and exact sequences}\label{sec:exactseq} 
    We assume that $N$ is a finite generated $R$-module $R^r \oplus R/(c_{r+1}) \oplus \dots \oplus R/(c_{r+t})$ generated by $\bar{\ebf}_1, \dots,  \bar{\ebf}_r, \bar{\ebf}_{r+1}, \dots, \bar{\ebf}_{r+t}$ where $\bar{\ebf}_1, \dots, \bar{\ebf}_r$ generate the free part of $N$. 
    For $S \subset E$, let $\ebf_{S, \iota} \coloneqq \bar{\ebf}_{\iota} + \rho(S)$. 
    Then, we can write the generators of $C^{S,j}(M, \rho)$ as the exterior, i.e., for $J = \{ \iota_1 < \cdots < \iota_{j} \} \subset [r + t]$,
    \begin{equation*}
        \ebf_{S, J} = \ebf_{S, \iota_1} \wedge \dots \wedge \ebf_{S, \iota_{j}}.
    \end{equation*}

    Let $e \in E$ be not a coloop. 
    Then, for $\ebf_{S, J} \in C^{i, j}(M/e, \rho_{/e})$ define $\alpha^{i, j}(\ebf_{S, J})$ by $(-1)^{\abs{S}} \epsilon^{S, e} \ebf_{S \cup e, J}$. 
    This gives a homomorphism $\alpha^{i, j}: C^{i, j}(M/e, \rho_{/e}) \to C^{i + 1, j}(M, \rho)$. 
    For $\ebf_{S, J} \in C^{i, j}(M, \rho)$ define $\beta^{i,j}(\ebf_{S,J})$ by $0$ if $e \in S$, and $\ebf_{S,J}$ if $e \notin S$. 
    This gives a homomorphism $\beta^{i, j}: C^{i, j}(M, \rho) \to C^{i, j}(M \setminus e, \rho_{\setminus e})$. 
    We abbreviate $\alpha^{i, j}$ and $\beta^{i, j}$ by $\alpha$ and $\beta$ for simplicity.

    Using the above homomorphisms we have the following proposition which essentially proceeds in the same way as in \cite{DL, HGR, JHR}.
    \begin{prop}\label{prop:shortex} 
    Suppose that $e$ is not a coloop. 
    Then, the maps $\alpha$ and $\beta$ are chain maps such that $0 \to C^{i-1,j}(M/e, \rho_{/e}) \xrightarrow{\alpha} C^{i,j}(M, \rho) \xrightarrow{\beta}C^{i, j}(M \setminus e, \rho_{\setminus e}) \to 0$ is a short exact sequence.
    \end{prop}

    \begin{proof}
        To begin with, let us show that $\alpha :C^{i-1,j}(M/e, \rho_{/e}) \to C^{i,j}(M, \rho)$ is a chain map.

        We have
        \begin{align*}
            \left( d_{(M, \rho)}\circ \alpha \right) (\ebf_{S,J}) & = d_{(M, \rho)}\left( (-1)^{\abs{S}}\epsilon^{S,e} \ebf_{S \cup e, J}\right)                                        \\
                                                                        & = \sum_{f \in E \setminus (S \cup e)}(-1)^{\abs{S}} \epsilon^{S, e}\epsilon^{S \cup e, f} \ebf_{S \cup \{ e, f \}, J}.
        \end{align*}
        On the other hand, we have
        \begin{align*}
            \left( \alpha \circ d_{(M/e, \rho_{/e})} \right) (\ebf_{S,J}) & = \alpha \left( \sum_{f \in E \setminus (S \cup e)} \epsilon^{S,f} \ebf_{S \cup f, J} \right)                              \\
                                                                              & = \sum_{f \in E \setminus (S \cup e)}(-1)^{\abs{S \cup f}} \epsilon^{S, f} \epsilon^{S \cup f, e} \ebf_{S \cup \{ e, f \}, J}.
        \end{align*}
        Since we have
        \begin{align*}
            \sum_{f \in E \setminus (S \cup e)} (-1)^{\abs{S \cup f}} \epsilon^{S, f} \epsilon^{S \cup f, e} \ebf_{S \cup \{ e, f \}, J} & = \sum_{f \in E \setminus (S \cup e)}(-1)^{\abs{S \cup f} + 1} \epsilon^{S, e} \epsilon^{S \cup e, f} \ebf_{S \cup \{ e, f \}, J} \\
                                                                                                                                         & = \sum_{f \in E \setminus (S \cup e)}(-1)^{\abs{S}} \epsilon^{S, e} \epsilon^{S \cup e, f} \ebf_{S \cup \{ e, f \}, J}
        \end{align*}
        the diagram
        \begin{center}
            \begin{tikzcd}
                C^{i-1,j}(M/e, \rho_{/e}) \ar[r, "\alpha"] \arrow[d, "d_{(M/e, \rho_{/e})}"'] & C^{i,j}(M, \rho) \ar[d, "d_{(M, \rho)}"] \\ C^{i,j}(M/e, \rho_{/e}) \ar[r, "\alpha"] & C^{i+1,j}(M, \rho)
            \end{tikzcd}
        \end{center}
        commutes. 
        Next, let us show that $\beta: C^{i,j}(M, \rho) \to C^{i,j}(M \setminus e, \rho_{\setminus e})$ is a chain map. \\ 
        When $e \in S$, the commutativity is obvious, so we can assume $e \notin S$. 
        In this case, we have
        \begin{align*}
            \left( d_{(M \setminus e, \rho_{\setminus e})} \circ \beta \right) (\ebf_{S,J}) & = d_{(M \setminus e, \rho_{\setminus e})}(\ebf_{S, J})                \\
                                                                                                 & = \sum_{f \in E \setminus (S \cup e)} \epsilon^{S, f} \ebf_{S \cup f, J}.
        \end{align*}
        On the other hand, we have
        \begin{align*}
            \left( \beta \circ d_{(M, \rho)} \right) (\ebf_{S,J}) & = \beta \left( \sum_{f \in E \setminus S} \epsilon^{S, f} \ebf_{S \cup f, J} \right)                                                                                                                                     \\
                                                                       & = \beta \left( \sum_{f \in E \setminus (S \cup e)}\epsilon^{S, f} \ebf_{S \cup f, J} \right) + \beta(\epsilon^{S, e} \ebf_{S \cup e, J}) = \sum_{f \in E \setminus (S \cup e)} \epsilon^{S, f} \ebf_{S \cup f, J}.
        \end{align*}
        Thus, the diagram
        \begin{center}
            \begin{tikzcd}
                C^{i,j}(M, \rho) \ar[r, "\beta"] \arrow[d, "d_{(M, \rho)}"'] & C^{i,j}(M \setminus e, \rho_{\setminus e}) \ar[d, "d_{(M \setminus e, \rho_{\setminus e})}"] \\ C^{i+1,j}(M, \rho) \ar[r, "\beta"] & C^{i+1,j}(M \setminus e, \rho_{\setminus e})
            \end{tikzcd}
        \end{center}
        commutes.
    
Let us next show the exactness of the sequence
\begin{equation*}
    0 \to C^{i-1,j}(M/e,\rho_{/e}) \xrightarrow{\alpha} C^{i,j}(M,\rho) \xrightarrow{\beta} C^{i,j}(M\setminus e,\rho_{\setminus e}) \to 0.
\end{equation*}

First, note that
\begin{equation*}
    C^{i,j}(M,\rho) = \bigoplus_{\substack{S \in \binom{E}{i} \\ e \notin S}} C^{S,j}(M,\rho) \oplus \bigoplus_{\substack{S \in \binom{E}{i} \\ e \in S}} C^{S,j}(M,\rho).
\end{equation*}
By definition, the map $\beta$ is the projection onto the first summand.
Hence $\beta$ is surjective and $\kerne \beta = \bigoplus_{S \in \binom{E}{i},\ e\in S} C^{S,j}(M,\rho)$.

On the other hand, for each subset $S \subset E \setminus e$ with $\abs{S} = i - 1$, we have
\begin{equation*}
C^{S,j}(M/e,\rho_{/e}) = \bigwedge^j \frac{\rho(E)/\rho(e)} {\rho(S\cup e)/\rho(e)} \cong \bigwedge^j \frac{\rho(E)} {\rho(S \cup e)} = C^{S \cup e,j}(M,\rho).
\end{equation*}

Under this identification, the map $\alpha$ restricts to an isomorphism
\begin{equation*}
C^{S,j}(M/e,\rho_{/e}) \longrightarrow C^{S \cup e,j}(M,\rho)
\end{equation*}
up to the sign $(-1)^{\abs{S}} \epsilon^{S,e}$.
Therefore $\alpha$ is injective and $\im \alpha = \bigoplus_{S\in\binom{E}{i},\ e\in S}C^{S,j}(M,\rho)$.

Consequently, $\im \alpha = \kerne \beta$, hence the sequence is exact.
\end{proof}

    By the Zig-Zag lemma we have a long exact sequence of the characteristic cohomologies of matroids which categorifies the second property in Proposition \ref{prop:CharPoly}.
    \begin{thm}\label{thm:LongExactSq} 
    Consider a matroid $(M, \rho) = (E, \rk_{M}, \rho)$ and fix an element $e \in E$ which is not a coloop. 
    Then, for each $j$, we have a long exact sequence
        \begin{align*}
            0 \to H^{0,j}(M, \rho) & \xrightarrow{\beta^\ast}H^{0,j}(M\setminus e, \rho_{\setminus e}) \xrightarrow{\gamma^\ast}H^{0, j}(M/e, \rho_{/e}) \xrightarrow{\alpha^\ast}H^{1,j}(M, \rho)      \\
                                   & \xrightarrow{\beta^\ast}H^{1,j}(M \setminus e, \rho_{\setminus e}) \xrightarrow{\gamma^\ast}H^{1, j}(M/e, \rho_{/e}) \xrightarrow{\alpha^\ast}\cdots               \\
                                   & \xrightarrow{\beta^\ast}H^{i,j}(M \setminus e, \rho_{\setminus e}) \xrightarrow{\gamma^\ast}H^{i, j}(M/e, \rho_{/e}) \xrightarrow{\alpha^\ast}H^{i + 1,j}(M, \rho) \\
                                   & \to \cdots
        \end{align*}
        By summing over $j$, we have a long exact sequence
        \begin{align*}
            0 \to H^{0}(M, \rho) & \xrightarrow{\beta^\ast}H^{0}(M \setminus e, \rho_{\setminus e}) \xrightarrow{\gamma^\ast}H^{0}(M/e, \rho_{/e}) \xrightarrow{\alpha^\ast}H^1(M, \rho)     \\
                                 & \xrightarrow{\beta^\ast}H^1(M \setminus e, \rho_{\setminus e}) \xrightarrow{\gamma^\ast}H^1(M/e, \rho_{/e}) \xrightarrow{\alpha^\ast} \cdots             \\
                                 & \xrightarrow{\beta^\ast}H^i(M \setminus e, \rho_{\setminus e}) \xrightarrow{\gamma^\ast}H^i(M/e, \rho_{/e}) \xrightarrow{\alpha^\ast}H^{i + 1}(M, \rho) \\
                                 & \to \cdots.
        \end{align*}
    \end{thm}

    If $e$ is a coloop, then we need to slightly modify the construction.
    Suppose $\rho(E) = \rho(E \setminus e) \oplus \rho(e)$ and $\rho(e)$ is a free module. 
    Then we retake generators $\ebf_1, \dots, \ebf_r,$ $\ebf_{r+1}, \dots, \ebf_{r+t}$ of $N$ again such that $\ebf_1, \dots, \ebf_{r-1}, \ebf_{r+1}, \dots, \ebf_{r+t}$ span $\rho(E \setminus e)$ and $\rho(e) = \left< \ebf_r\right>$. 
    Since $\rho(E)= \rho(E\setminus e)\oplus\rho(e)$, we get $\frac{\rho(E)}{\rho(S)}= \frac{\rho(E\setminus e)}{\rho(S)} \oplus\rho(e)$ for $S\subset E\setminus e$. 
    Thus we obtain $C^{S,j}(M, \rho)\cong C^{S,j}(M \setminus e, \rho_{\setminus e})\oplus C^{S,j-1}(M \setminus e, \rho_{\setminus e})\otimes \left<\ebf_r\right>\cong C^{S,j}(M \setminus e, \rho_{\setminus e})\oplus C^{S,j-1}(M \setminus e, \rho_{\setminus e})$ for $j>0$. 
    Modify $\beta$ as follows:
    \begin{align*}
        \beta: & C^{i,j}(M, \rho) \to C^{i,j}(M \setminus e, \rho_{\setminus e})\oplus C^{i,j-1}(M \setminus e, \rho_{\setminus e}) \text{ via}                                                                                                                                                                                                   \\
               & \beta(\ebf_{S,J})= \left\{ \begin{aligned}&0&&\mbox{ if }e\in S,\\&\ebf_{S,J}\in C^{i,j}(M \setminus e, \rho_{\setminus e})&&\mbox{ if }e \notin S \mbox{ and }r \notin J,\\&\ebf_{S,J\setminus r}\in C^{i,j-1}(M \setminus e, \rho_{\setminus e})&&\mbox{ if }e\notin S\mbox{ and }r\in J.\end{aligned}\right.
    \end{align*}
    Then, we obtain a short exact sequence $0 \to C^{i-1,j}(M/e, \rho_{/e}) \xrightarrow{\alpha}C^{i,j}(M, \rho) \xrightarrow{\beta}C^{i,j}(M \setminus e, \rho_{\setminus e})\oplus C^{i,j-1}(M \setminus e, \rho_{\setminus e}) \to 0$. 
    By the Zig-Zag lemma we have a long exact sequence of the characteristic cohomologies of matroids which categorifies the third property in Proposition \ref{prop:CharPoly}.
    \begin{thm}\label{thm:coloop} 
    For a matroid $(M, \rho) = (E, \rk_{M},\rho)$ and a coloop $e \in E$, if $\rho(E) = \rho(E\setminus e) \oplus \rho(e)$ with $\rho(e)$ a free module, then for each $j>0$, we have a long exact sequence
        \begin{align*}
             & 0 \to H^{0,j}(M, \rho) \xrightarrow{\beta^\ast}H^{0,j}(M \setminus e, \rho_{\setminus e})\oplus H^{0,j-1}(M \setminus e, \rho_{\setminus e}) \xrightarrow{\gamma^\ast}H^{0, j}(M/e, \rho_{/e})                 \\
             & \xrightarrow{\alpha^\ast}H^{1,j}(M, \rho) \xrightarrow{\beta^\ast}H^{1,j}(M \setminus e, \rho_{\setminus e})\oplus H^{1,j-1}(M \setminus e, \rho_{\setminus e}) \xrightarrow{\gamma^\ast}H^{1, j}(M/e, \rho_{/e}) \\
             & \xrightarrow{\alpha^\ast}\cdots \xrightarrow{\beta^\ast}H^{i,j}(M \setminus e, \rho_{\setminus e})\oplus H^{i,j-1}(M \setminus e, \rho_{\setminus e}) \xrightarrow{\gamma^\ast}H^{i, j}(M/e, \rho_{/e})           \\
             & \xrightarrow{\alpha^\ast}H^{i + 1,j}(M, \rho) \to \cdots .
        \end{align*}
    \end{thm}

    The following corollary is a matroid version of the statement of Theorem 3.9 in \cite{HGR}.
    \begin{cor}\label{cor:coloop} 
    Under the assumption of the above theorem, we have $H^{i,j}(M, \rho)\cong H^{i,j-1}(M \setminus e, \rho_{\setminus e})\cong H^{i,j-1}(M/e, \rho_{/e})$.
    \end{cor}
    \begin{proof}
        Since, by assumption, $\rho_{\setminus e}(S) = \rho(S) \cong(\rho(S) + \rho(e))/\rho(e) = \rho_{/e}(S)$ for all $S \subset E \setminus e$, we have $(M \setminus e, \rho_{\setminus e}) \cong (M/e, \rho_{/e})$ as a matroid and a quasi-representation. 
        Thus we obtain $H^{i,j}(M \setminus e, \rho_{\setminus e}) \cong H^{i,j}(M/e, \rho_{/e})$ for each $i,j$.

        If $\gamma^\ast$ induces the isomorphism, we have $\alpha^\ast = 0$ and also obtain $\beta^\ast$ is injective. 
        By restricting the codomain to $H^{i,j-1}(M \setminus e, \rho_{\setminus e})$, we obtain isomorphism $\beta^\ast: H^{i,j}(M, \rho) \to H^{i,j-1}(M \setminus e, \rho_{\setminus e})$.

        In the rest of the proof we show that the connecting homomorphism $\gamma^\ast$ induces the isomorphism. \\ 
        Take a non-zero cocycle
        \begin{equation*}
            z = \sum_{\substack{S \in \binom{E}{i} \\ r\in J}}a_{S, J} \ebf_{S, J\setminus r}+\sum_{\substack{S \in \binom{E}{i} \\ r \notin J}}a_{S, J} \ebf_{S, J}\in \kerne d^{i,j-1}_{M \setminus e, \rho_{\setminus e}} \oplus \kerne d^{i,j}_{M \setminus e, \rho_{\setminus e}}.
        \end{equation*}
        By definition of the map $\beta$ and its surjectivity, we have $w = \sum_{S,J}a_{S, J} \ebf_{S, J}\in C^{i,j}(M, \rho)$ such that $\beta(w)=z$ and the commutativity gives
        \begin{align*}
            d^{i,j}_{(M, \rho)}(w) & =  \sum_{S,J}a_{S,J}\sum_{f \notin S}\epsilon^{S,f} \ebf_{S \cup f, J}                                                                    \\
                                     & = \sum_{S,J}a_{S,J}\epsilon^{S,e} \ebf_{S \cup e, J}+ \sum_{S,J}a_{S,J}\sum_{f \notin S \cup e}\epsilon^{S, f} \ebf_{S \cup f, J} \\
                                     & = \sum_{S,J}a_{S,J}\epsilon^{S,e} \ebf_{S \cup e, J}                                                                                   \\
                                     & = \sum_{S,r \notin J}a_{S,J}\epsilon^{S,e} \ebf_{S \cup e, J}.
        \end{align*}
        In the final row, we use $\ebf_{S\cup e,J}=0$ for $r\in J$. 
        Then, by definition of the map $\alpha$ and the exactness of the sequence
        \begin{equation*}
            0 \to C^{i,j}(M/e, \rho_{/e}) \xrightarrow{\alpha}C^{i+1,j}(M, \rho) \xrightarrow{\beta}C^{i+1,j}(M \setminus e, \rho_{\setminus e}) \oplus C^{i+1,j-1}(M \setminus e, \rho_{\setminus e}) \to 0,
        \end{equation*}
        we have
        \begin{equation*}
            w' = \sum_{\substack{S \in \binom{E}{i} \\ r \notin J}}a_{S,J}(-1)^{\abs{S}} \ebf_{S, J}=(-1)^i\sum_{\substack{S \in \binom{E}{i} \\ r \notin J}}a_{S,J} \ebf_{S, J}\in C^{i,j}(M/e, \rho_{/e})
        \end{equation*}
        with $\alpha(w') = d^{i,j}_{(M, \rho)}(w)$. 
        Note that the connecting homomorphism $\gamma^\ast: H^{i,j-1}(M \setminus e, \rho_{\setminus e}) \oplus H^{i,j}(M \setminus e, \rho_{\setminus e}) \to H^{i,j}(M/e, \rho_{/e})$ sends any elements in $H^{i,j-1}(M \setminus e, \rho_{\setminus e})$ to zero. 
        Therefore, by restricting the domain to $H^{i,j}(M \setminus e, \rho_{\setminus e})$, the connecting homomorphism $\gamma^\ast$ is $(-1)^i$ times the canonical isomorphism
        \begin{equation*}
            H^{i,j}(M \setminus e, \rho_{\setminus e}) \xrightarrow{\cong}H^{i,j}(M/e, \rho_{/e}).
        \end{equation*}
    \end{proof}

    \subsection{Loops, parallel elements, direct sum and relaxation}
    In this subsection, let us categorify the properties of loops, parallel elements, direct sum, and relaxation in matroid theory. 
    We begin with categorifying the fourth property in Proposition \ref{prop:CharPoly}, which is analogous to Propositions 3.4 in \cite{HGR} for the chromatic cohomology.
    \begin{prop}\label{prop:loop} 
    If the matroid with a quasi-representation $(M,\rho)$ has a loop, then $H^i(M,\rho) = 0$ for all $i$.
    \end{prop}
\begin{proof}
Let $e$ be a loop of $(M,\rho)$.
Since $e$ is a loop, we have $\rk_M(S)= \rk_M(S\setminus e)$ for every subset $S\subset E$ containing $e$.
Hence, by the definition of a quasi-representation, $\rho(S)= \rho(S\setminus e)$.
Therefore, we have
\begin{equation*}
    C^{S,j}(M,\rho)=C^{S\setminus e,j}(M,\rho)    
\end{equation*}
whenever $e\in S$.
Let $h^{S,j}: C^{S,j}(M,\rho) \to C^{S\setminus e,j}(M,\rho)$ be a map defined by multiplying by $\epsilon^{S,e}$ if $S\ni e$, and $h^{S,j}$ be zero map if $S\not\ni e$.
We define an $R$-module homomorphism $h^{i,j}:C^{i,j}(M,\rho) \to C^{i-1,j}(M,\rho)$ by summing up $h^{S,j}$, that is
\begin{equation*}
    h^{i,j}(\mathbf e_{S,J})=
    \begin{cases}
    \epsilon^{S,e}\mathbf e_{S\setminus e,J},
    & e\in S,\\[1ex]
    0,
    & e\notin S.
    \end{cases}
\end{equation*}

A direct computation shows that
\begin{equation*}
    d^{i-1,j}h^{i,j}+h^{i+1,j}d^{i,j}= \mathrm{id}_{C^{i,j}(M,\rho)}.
\end{equation*}
First assume that $e \notin S$.
 Then $h^{i,j}(\mathbf e_{S,J})=0$, hence we have 
 \begin{equation*}
 (d^{i-1,j}h^{i,j}+h^{i+1,j}d^{i,j})(\mathbf e_{S,J}) = h^{i+1,j} \left( \sum_{f \notin S} \epsilon^{S,f} \mathbf e_{S \cup f,J} \right) = \epsilon^{S \cup e, e} \epsilon^{S, e} \mathbf e_{S,J}.
 \end{equation*}
Since $\epsilon^{S \cup e, e} = \epsilon^{S, e}$, we get $(d^{i-1,j}h^{i,j}+h^{i+1,j}d^{i,j})(\ebf_{S,J}) = \ebf_{S,J}$.

Next assume that $e \in S$.
Then 
\begin{align*}
&\left( d^{i-1,j} h^{i,j} + h^{i+1,j} d^{i,j} \right) (\ebf_{S,J}) \\
& = d^{i-1,j} \left( \epsilon^{S,e} \ebf_{S \setminus e, J} \right) + h^{i+1,j} \left( \sum_{f \notin S} \epsilon^{S,f} \ebf_{S \cup f, J} \right)  \\
& = \epsilon^{S,e} \sum_{f \notin S \setminus e} \epsilon^{S \setminus e, f} \ebf_{(S \setminus e) \cup f, J} + \sum_{f \notin S} \epsilon^{S, f} \epsilon^{S\cup f,e} \ebf_{(S \setminus e) \cup f, J} \\
& = \sum_{f \notin S} (\epsilon^{S,e}\epsilon^{S\setminus e,f}+\epsilon^{S\cup f,e}\epsilon^{S,f}) \ebf_{(S \setminus e) \cup f, J} + \ebf_{S,J}
\end{align*}
It is easy to see that
\begin{alignat*}{4}
    \epsilon^{S,e}&= \epsilon^{S\cup f,e},&\quad& \epsilon^{S\setminus e,f}&=-\epsilon^{S,f},&\qquad&\text{if }e<f,\\
    \epsilon^{S,e}&=-\epsilon^{S\cup f,e},&\quad& \epsilon^{S\setminus e,f}&= \epsilon^{S,f},&\qquad&\text{if }e>f.
\end{alignat*}
Hence, we have $\epsilon^{S,e}\epsilon^{S\setminus e,f}+\epsilon^{S\cup f,e}\epsilon^{S,f}=0$ which implies $\left( d^{i-1,j} h^{i,j} + h^{i+1,j} d^{i,j} \right) (\ebf_{S,J})= \ebf_{S,J}$.

Thus the identity chain map $\mathrm{id}:C^{S,j}(M,\rho) \to C^{S,j}(M,\rho)$ is chain homotopic to zero.
Hence the chain complex $C^{\bullet,j}(M,\rho)$ is contractible for every $j$, and therefore $H^{i,j}(M,\rho)=0$ for all $i,j$.
Consequently, $H^i(M,\rho)=0$ for all $i$.
\end{proof}
    The following proposition categorifies the fifth property in Proposition \ref{prop:CharPoly}, which is analogous to Propositions 3.5 in \cite{HGR} for the chromatic cohomology.
    \begin{prop}\label{PropParallel} 
    If the matroid with quasi-representation $(M,\rho)$ has parallel elements $e_1,e_2$, then $H^i(M,\rho) \simeq H^i(M\setminus e_1,\rho_{\setminus e_1})\simeq H^i(M\setminus e_2,\rho_{\setminus e_2})$ for all $i$.
    \end{prop}
    \begin{proof}
        Let $(M,\rho)$ have parallel elements $e_1, e_2$. 
        Then, consider the long exact sequence given by Theorem \ref{thm:LongExactSq}.
        \begin{align*}
            0 \to H^{0}(M,\rho) & \xrightarrow{\beta^\ast}H^{0}(M\setminus e_1,\rho_{\setminus e_1}) \xrightarrow{\gamma^\ast}H^{0}(M/e_1,\rho_{/e_1}) \xrightarrow{\alpha^\ast}H^1(M,\rho)       \\
                                & \xrightarrow{\beta^\ast}H^1(M\setminus e_1,\rho_{\setminus e_1}) \xrightarrow{\gamma^\ast}H^1(M/e_1,\rho_{/e_1}) \xrightarrow{\alpha^\ast}\cdots               \\
                                & \xrightarrow{\beta^\ast}H^i(M\setminus e_1,\rho_{\setminus e_1}) \xrightarrow{\gamma^\ast}H^i(M/e_1,\rho_{/e_1}) \xrightarrow{\alpha^\ast}H^{i + 1}(M,\rho) \to \cdots.
        \end{align*}
        Note that since $\rk_{M/e_1}(e_2) = \rk_{M}(\{e_1,e_2 \}) - \rk_{M}(e_1) = 1-1 = 0$, the element $e_2$ is a loop in the contraction $(M/e_1,\rho_{/e_1})$. 
        By Proposition \ref{prop:loop}, $H^i(M/e_1,\rho_{/e_1}) = 0$ for all $i$, and thus we have $H^i(M,\rho) \simeq H^i(M\setminus e,\rho_{\setminus e})$ for all $i$.
    \end{proof}
    The following proposition categorifies the sixth property in Proposition \ref{prop:CharPoly}.
    \begin{prop}\label{prop:Kunneth} 
    Let $(M,\rho)$ and $(M',\rho')$ be matroids with quasi-representations. 
    Then $C^{\bullet}((M,\rho)\oplus (M',\rho')) = C^{\bullet}(M,\rho)\otimes C^{\bullet}(M',\rho')$ as graded chain groups. 
    In particular, if $R= \Bbbk$ is a field, then $H^{i,\bullet}((M,\rho)\oplus (M',\rho')) = \bigoplus_{l=0}^iH^{i-l,\bullet}(M,\rho)\otimes H^{l,\bullet}(M',\rho')$.
    \end{prop}
    \begin{proof}
        Let $E$ and $E'$ be ground sets of $M$ and $M'$. Suppose that the order on $E\sqcup E'$ is defined by $e<e'$ for $e\in E$ and $e'\in E'$. 
        Note that, for each subset $S \subset E$ and $S' \subset E'$,
        \begin{align*}
            C^{S\sqcup S',j}((M,\rho)\oplus (M',\rho')) & = \bigwedge^{j}\frac{\rho(E)\oplus \rho'(E')}{\rho(S)\oplus\rho'(S')}\cong \bigoplus_{j'=0}^{j}\bigwedge^{j-j'}\frac{\rho(E)}{\rho(S)}\otimes\bigwedge^{j'}\frac{\rho'(E')}{\rho'(S')} \\
                                                        & = \bigoplus_{j'=0}^{j}C^{S,j-j'}(M,\rho)\otimes C^{S',j'}(M',\rho')
        \end{align*}
        as graded modules. 
        Regarding the differential, for $x\in C^{S,j-j'}(M,\rho)$ and $x' \in C^{S',j'}$ $(M',\rho')$, we need to check that $d_{(M,\rho)\oplus (M',\rho')}^{\abs{S}+\abs{S'},j}(x\otimes x')=d_{(M,\rho)}(x)\otimes x'+(-1)^{\abs{S}}x\otimes d_{(M',\rho')}(x')$. 
        However, this relation is obtained from the definition of differential by simple computation. Hence $C^{\bullet}((M,\rho)\oplus (M',\rho')) = C^{\bullet}(M,\rho)\otimes C^{\bullet}(M',\rho')$ as graded chain groups. 
        When $R= \Bbbk$, the proposition follows from the K\"unneth theorem.
    \end{proof}

    We close the subsection with the categorification of the last property in Proposition \ref{prop:CharPoly}.

    \begin{thm}\label{thm:relax}
    Let $(M,\rho)$ be a matroid of rank $r$, and let $(M^{-},\rho^{-})$ be its relaxation obtained by relaxing a circuit-hyperplane $S_0$.
    Then
    \begin{equation*}
    H^{i,j}(M,\rho)\cong H^{i,j}(M^{-},\rho^{-})
    \end{equation*}
    for $i\neq r-1$.
    Moreover, for every $j\geq 1$, there is a short exact sequence
    \begin{equation*}
    0\longrightarrow H^{r-1,j}(M,\rho) \longrightarrow H^{r-1,j}(M^{-},\rho^{-}) \longrightarrow \bigwedge^j \left(N/\rho(S_0)\right) \longrightarrow 0.
    \end{equation*}
    In particular, if $N/\rho(S_0)$ is a free $R$-module, the above sequence splits, and hence
    \begin{align*}
    H^{r-1,1}(M^{-},\rho^{-}) &\cong H^{r-1,1}(M,\rho)\oplus N/\rho(S_0), \\
    H^{r-1,j}(M^{-},\rho^{-}) &\cong H^{r-1,j}(M,\rho) \qquad (j\neq 1).
    \end{align*}
    \end{thm}
    \begin{proof}
    When $j=0$, for every $S\subset E$, we have
    \begin{equation*}
    C^{S,0}(M,\rho) = \bigwedge^0\left(N/\rho(S)\right)=R= \bigwedge^0\left(N/\rho^{-}(S)\right)=C^{S,0}(M^{-},\rho^{-}).
    \end{equation*}
    Moreover, the differentials in degree $j=0$ agree, since every induced map on the zeroth exterior power is the identity map on $R$. 
    Hence we have $C^{\bullet,0}(M,\rho)=C^{\bullet,0}(M^{-},\rho^{-})$, thus $H^{i,0}(M,\rho)\cong H^{i,0}(M^{-},\rho^{-})$ for all $i$.
    Hereafter we assume $j\ge 1$.
    
    The only difference between $\rho$ and $\rho^{-}$ occurs at the subset $S_0$, where $\rho^{-}(S_0)=N$. Hence, for $S\neq S_0$, we have $C^{S,j}(M,\rho)=C^{S,j}(M^{-},\rho^{-})$.
    Since $\abs{S_0} = r$, the chain groups differ only in degree $r$.
    Moreover, the differentials differ only for $d^{r-1,j}:C^{r-1,j}\to C^{r,j}$ and $d^{r,j}:C^{r,j}\to C^{r+1,j}$.
    Therefore
    \begin{equation*}
        H^{i,j}(M,\rho)\cong H^{i,j}(M^{-},\rho^{-})
    \end{equation*}
    for $i\neq r-1,r,r+1$.
    
    First, we show that the cohomology groups also agree for $i=r$ and $i=r+1$.
    Since $S_0$ is a circuit-hyperplane, for every $e\in E\setminus S_0$ we have $\rho(S_0\cup e)= \rho^{-}(S_0\cup e)=N$.
    Hence we have $C^{S_0\cup e,j}(M,\rho)=C^{S_0\cup e,j}(M^{-},\rho^{-})=0$.
    Thus, we have $\im d^{r,j}_{(M,\rho)} = \im d^{r,j}_{(M^-,\rho^-)}$.
    Moreover,
    \begin{equation*}
        C^{r,j}(M,\rho) = C^{r,j}(M^{-},\rho^{-})\oplus \bigwedge^j \left( N/\rho(S_0) \right).
    \end{equation*}
    
    The differential $d^{r,j}_{(M,\rho)}$ vanishes on $C^{S_0,j}(M,\rho) = \bigwedge^j\left(N/\rho(S_0)\right)$.
    Indeed, for every $e \in E \setminus S_0$, we have $\rho(S_0\cup e) = \rho(E)$ and therefore $C^{S_0\cup e,j}(M,\rho)=0$.
    Thus, we have
    \begin{equation*}
        \kerne d^{r,j}_{(M,\rho)}
        = \kerne d^{r,j}_{(M^{-},\rho^{-})} \oplus \bigwedge^j \left( N/\rho(S_0) \right).
    \end{equation*}
    On the other hand, $\bigwedge^j\left(N/\rho(S_0)\right) \subset \im d^{r-1,j}_{(M,\rho)}$.
    Indeed, for any $e_0\in S_0$ the map $d^{S_0\setminus e_0,e_0}: C^{S_0\setminus e_0,j}(M,\rho) \longrightarrow C^{S_0,j}(M,\rho) = \bigwedge^j\left(N/\rho(S_0)\right)$ is the natural projection, up to the sign $\epsilon^{S_0\setminus e_0,e_0}$, and hence is surjective.
    Since the differentials $d^{r-1,j}_{(M,\rho)}$ and $d^{r-1,j}_{(M^{-},\rho^{-})}$ coincide on all components except those involving $C^{S_0,j}(M,\rho) = \bigwedge^j\left(N/\rho(S_0)\right)$, we obtain
    \begin{equation*}
    \im d^{r-1,j}_{(M,\rho)} = \im d^{r-1,j}_{(M^{-},\rho^{-})} \oplus \bigwedge^j\left(N/\rho(S_0)\right).
    \end{equation*}
    Consequently,
    \begin{equation*}
        H^{r,j}(M,\rho)\cong H^{r,j}(M^{-},\rho^{-}).
    \end{equation*}
    The same equality of images in degree $r+1$ gives
    \begin{equation*}
        H^{r+1,j}(M,\rho)\cong H^{r+1,j}(M^{-},\rho^{-}).
    \end{equation*}
    
    It remains to check degree $r-1$.
    Since the differentials
    $d^{r-2,j}_{(M,\rho)}$ and $d^{r-2,j}_{(M^{-},\rho^{-})}$ coincide, we have $\im d^{r-2,j}_{(M,\rho)}  = \im d^{r-2,j}_{(M^{-},\rho^{-})}$.
    
    Let
    \begin{equation*}
        q: \kerne d^{r-1,j}_{(M^{-},\rho^{-})} \hookrightarrow C^{r-1,j}(M^-,\rho^-)=C^{r-1,j}(M,\rho)\xrightarrow{d_M} C^{r,j}(M,\rho)\twoheadrightarrow \bigwedge^j\left(N/\rho(S_0)\right)
    \end{equation*}
    where $d_M=d^{r-1,j}_{(M,\rho)}$ and the last map is the natural projection $C^{r,j}(M,\rho) \to C^{S_0,j}(M,\rho)$.
    To show the exactness of the sequence
    \begin{equation}\label{eq:relax}
    0 \longrightarrow H^{r-1,j}(M,\rho) \longrightarrow H^{r-1,j}(M^{-},\rho^{-}) \longrightarrow \bigwedge^j \left( N/\rho(S_0) \right) \longrightarrow 0,
    \end{equation}
    we show that
    \begin{equation}\label{eq:relax2}
        0\to \kerne d^{r-1,j}_{(M,\rho)}\hookrightarrow \kerne d^{r-1,j}_{(M^{-},\rho^{-})}\xrightarrow{q} \bigwedge^j\left(N/\rho(S_0)\right) \to 0
    \end{equation}
    is exact.
    
    Since the maps $d^{r-1,j}_{(M,\rho)}$ and $d^{r-1,j}_{(M^{-},\rho^{-})}$ differ only in the map
    \begin{equation*}
    d^{S_0\setminus e,e}: C^{S_0\setminus e,j}(M,\rho) \to C^{S_0,j}(M,\rho) = \bigwedge^j\left(N/\rho(S_0)\right)
    \end{equation*}
    for every $e\in S_0$, we have $\kerne d^{r-1,j}_{(M,\rho)} = \kerne q$.
    
    Now, we show $q$ is surjective.
    Let $\bar{x}\in \bigwedge^j \left( N/\rho(S_0) \right)$. 
    Choose $e_0\in S_0$ and take an element $x$ in the preimage of $\bar{x}$ under the natural projection $\bigwedge^j\left(N/\rho(S_0\setminus e_0)\right)=C^{S_0\setminus e_0,j}(M^{-},\rho^{-}) \to \bigwedge^j\left(N/\rho(S_0)\right)$.
    We claim that $x$ is a cocycle in $C^{\bullet,j}(M^{-},\rho^{-})$.
    Indeed, if we add $e_0$ to $S_0\setminus e_0$, then
    \begin{equation*}
        C^{S_0,j}(M^{-},\rho^{-}) = \bigwedge^j(N/\rho^{-}(S_0)) = 0.
    \end{equation*}
    If we add an element $e\in E\setminus S_0$, then, since $S_0$ is a circuit-hyperplane, we have $\rho^{-}((S_0\setminus e_0)\cup e) = N$, which implies 
    \begin{equation*}
        C^{(S_0\setminus e_0)\cup e,j}(M^{-},\rho^{-}) = \bigwedge^j(N/\rho^{-}((S_0\setminus e_0)\cup e)) = 0.
    \end{equation*}
    Hence we have $x\in \kerne d^{r-1,j}_{(M^{-},\rho^{-})}$.
    By construction, the image of $x$ in $\bigwedge^j \left( N/\rho(S_0) \right)$ is
    $\pm\bar{x}$, and thus $q$ is surjective.
    
    Moreover, we have $\im d^{r-2,j}_{(M^{-},\rho^{-})} \subset \kerne q$.
    Indeed, since $d^{r-2,j}_{(M,\rho)} = d^{r-2,j}_{(M^{-},\rho^{-})}$ and $d^{r-1,j}_{(M,\rho)} \circ d^{r-2,j}_{(M,\rho)} = 0$, we have $q\circ d^{r-2,j}_{(M^{-},\rho^{-})} = 0$.
    Thus $q$ induces a homomorphism 
    \begin{equation*} 
    \bar q: H^{r-1,j}(M^{-},\rho^{-}) \longrightarrow \bigwedge^j\left(N/\rho(S_0)\right). 
    \end{equation*}
    Since \eqref{eq:relax2} is exact and $\im d^{r-2,j}_{(M,\rho)} = \im d^{r-2,j}_{(M^{-},\rho^{-})}$, we obtain the desired sequence \eqref{eq:relax}.
    
    Finally, if $N/\rho(S_0)$ is free, then, since $S_0$ is a circuit-hyperplane, we have $N/\rho(S_0)\cong R$.
    Hence the short exact sequence for $j=1$ splits, and 
    \begin{align*}
    H^{r-1,1}(M^{-},\rho^{-}) &\cong H^{r-1,1}(M,\rho)\oplus N/\rho(S_0),\\
    H^{r-1,j}(M^{-},\rho^{-}) &\cong H^{r-1,j}(M,\rho) \qquad (j\neq 1).\qedhere
    \end{align*}
    \end{proof}

    \section{Sample computations}\label{sec:computation} 
    In this section we will calculate examples of uniform matroids. Recall that a uniform matroid $U_{k,n}$ is defined on the $n$ element set $E= \{ 1,\dots,n\}$ by the rank function $\rk_{U_{k,n}}(S)= \min(k,\abs{S})$. 
    Also note that, for any matroid $M$ on $E$ and quasi-representation, $C^{i,0}(M,\rho)= \bigoplus_{S\in \binom{E}{i}}R$ for $0\leq i\leq\abs{E}$. 
    Hence if $M\neq U_{0,0}$, then $H^{i,0}(M,\rho)=0$ for all $i$.

    \subsection{Canonical quasi-representations for uniform matroids}
    Among uniform matroids, the regular matroids are $U_{0,n},U_{1,n},U_{n-1,n},U_{n,n}$. 
    Throughout this subsection, we assume quasi-representation is canonical and simply write $M$ instead of $(M,\rho)$.

    \begin{prop}\label{prop:sample-loops} 
    Let $n$ be an integer greater than $0$. 
    Then,
        \begin{align*}
            H^{i,j}(U_{0,0}) & \cong \left\{\begin{aligned}R&&\mbox{ if } i = j = 0 \\ 
            0 & & \mbox{ otherwise,}\end{aligned}\right. \\
            H^{i,j}(U_{0,n}) & \cong 0 \quad \text{for all } i,j.
        \end{align*}
    \end{prop}
    \begin{proof}
        Since $C^{i,j}(U_{0,0})=0$ for all $(i,j) \neq (0,0)$, it follows that $H^{0,0}(U_{0,0})=R$ and $H^{i,j}(U_{0,0})=0$ for all $(i,j) \neq (0,0)$. 
        The rest follows from Proposition \ref{prop:loop}.
    \end{proof}

    \begin{prop}\label{ExamU1n} 
    Let $n$ be an integer greater than $0$. 
    Then,
        \begin{equation*}
            H^{i,j}(U_{1,n})\cong \left\{
            \begin{aligned}
                R &  & \mbox{ if }i=0,j=1 \\
                0 &  & \mbox{ otherwise.}
            \end{aligned}\right.\\
        \end{equation*}
    \end{prop}
    \begin{proof}
        By Proposition \ref{PropParallel}, it is sufficient to prove the case $n=1$. 
        Since $C^{0,0}(U_{1,1})=C^{1,0}(U_{1,1})=C^{0,1}(U_{1,1})=R$ and $C^{i,j}(U_{1,1})=0$ for the other $(i,j)$, the claim follows. 
        Hence we obtain $H^{0,1}(U_{1,1})=R$ and $H^{i,j}(U_{1,1})=0$ for all $(i,j) \neq (0,1)$.
    \end{proof}

    \begin{prop}\label{ExamUnn} 
    Let $n$ be an integer greater than $0$. 
    Then,
        \begin{equation*}
            H^{i,j}(U_{n,n})\cong \left\{
            \begin{aligned}
                R &  & \mbox{ if }i=0,j=n \\
                0 &  & \mbox{ otherwise.}
            \end{aligned}\right.\\
        \end{equation*}
    \end{prop}
    \begin{proof}
        Since $U_{n,n}$ consists of $n$ coloops, the proof follows from Corollary \ref{cor:coloop} and Proposition \ref{ExamU1n}.
    \end{proof}

    \begin{prop}
        Let $n$ be an integer greater than $0$. 
        Then,
        \begin{equation}
            \label{eq:circuit-hypo}H^{i,j}(U_{n-1,n})\cong \left\{
            \begin{aligned}
                R &  & \mbox{ if }i+j=n-1\mbox{ and }i\geq 0, j>0 \\
                0 &  & \mbox{ otherwise.}
            \end{aligned}\right.\\
        \end{equation}
    \end{prop}
    \begin{proof}
        We show the result by induction on $n$. 
        By Proposition \ref{prop:sample-loops}, $H^{i,j}(U_{0,1})=0$ for all $i,j$. 
        Next, suppose \eqref{eq:circuit-hypo} holds for some $n$. 
        Note that $U_{n,n+1}\setminus e\cong U_{n,n}$ and $U_{n,n+1}/ e\cong U_{n-1,n}$ for all elements $e$ in $U_{n,n+1}$. 
        By Theorem \ref{thm:LongExactSq}, we obtain a long exact sequence;
        \begin{align*}
            0 \to H^{0,j}(U_{n,n+1}) & \xrightarrow{\beta^\ast}H^{0,j}(U_{n,n}) \xrightarrow{\gamma^\ast}H^{0, j}(U_{n-1,n}) \xrightarrow{\alpha^\ast}H^{1,j}(U_{n,n+1})                 \\
                                     & \xrightarrow{\beta^\ast}H^{1,j}(U_{n,n}) \xrightarrow{\gamma^\ast}H^{1, j}(U_{n-1,n}) \xrightarrow{\alpha^\ast}\cdots                             \\
                                     & \xrightarrow{\beta^\ast}H^{i,j}(U_{n,n}) \xrightarrow{\gamma^\ast}H^{i, j}(U_{n-1,n}) \xrightarrow{\alpha^\ast}H^{i + 1,j}(U_{n,n+1}) \to \cdots.
        \end{align*}
        By Proposition \ref{ExamUnn} and the induction hypothesis, we obtain the following exact sequences
        \begin{align*}
             & 0 \to H^{0,j}(U_{n,n+1}) \xrightarrow{\beta^\ast} 0 \xrightarrow{\gamma^\ast} 0 \xrightarrow{\alpha^\ast}H^{1,j}(U_{n,n+1}) \xrightarrow{\beta^\ast} 0 & \mbox{ for }j<n-1,            \\
             & 0 \to H^{0,j}(U_{n,n+1}) \xrightarrow{\beta^\ast} 0 \xrightarrow{\gamma^\ast}R \xrightarrow{\alpha^\ast}H^{1,j}(U_{n,n+1}) \xrightarrow{\beta^\ast} 0 & \mbox{ for }j=n-1,            \\
             & 0 \to H^{0,n}(U_{n,n+1}) \xrightarrow{\beta^\ast}R \xrightarrow{\gamma^\ast} 0 \xrightarrow{\alpha^\ast}H^{1,n}(U_{n,n+1}) \xrightarrow{\beta^\ast} 0 & \mbox{ for }j=n,              \\
             & 0 \xrightarrow{\gamma^\ast}H^{i,j}(U_{n-1,n}) \xrightarrow{\alpha^\ast}H^{i + 1,j}(U_{n,n+1}) \xrightarrow{\beta^\ast} 0                          & \mbox{ for }i\geq 1, j> 0.
        \end{align*}
        Thus we obtain $H^{0,j}(U_{n,n+1})=0$ for $j\leq n-1$, 
        $H^{1,j}(U_{n,n+1})\cong 0$ for $j\neq n-1$, $H^{1,n-1}(U_{n,n+1})\cong R$, $H^{0,n}(U_{n,n+1}) \cong R$ and $R\cong H^{i + 1,j}(U_{n,n+1})$ for $i + j + 1 = n, i > 1, j > 0$.
    \end{proof}

    \subsection{Non-canonical quasi-representations for $U_{2,2}$}
    In this subsection, we fix a coefficient ring to the integer ring $\Zb$ and construct an example in which torsion appears.

    \begin{prop}
        Let $M$ be a uniform matroid $U_{2,2}$ on $E= \{ 1,2 \}$ and let $a,b$ be non-zero integers. 
        Define the $\Zb$-quasi-representation $\rho_{(a,b)}$ by
        \begin{align*}
             & N= \rho_{(a,b)}(E) = \Zb^2,        &  & \rho_{(a,b)}(\emptyset) = 0,               \\
             & \rho_{(a,b)}(\{ 1\}) = \left<(a,0)\right>, &  & \rho_{(a,b)}(\{2 \}) = \left<(0,b)\right>.
        \end{align*}
        Then
        \begin{align*}
            H^{1,1}(M,\rho_{(a,b)}) & \cong \Zb/a\Zb \oplus \Zb/b \Zb, \\
            H^{0,2}(M,\rho_{(a,b)}) & \cong \Zb,                       \\
            H^{1,2}(M,\rho_{(a,b)}) & \cong \Zb/\gcd(a,b)\Zb,          \\
        \end{align*}
        and the rest vanish.
    \end{prop}
    \begin{proof}
        It is sufficient to check the cases $0\leq i\leq 2$ and $0\leq j\leq 2$. 
        Since the case $j=0$ is obvious, first, let us check the case $j=1$. 
        By the definition,
        \begin{align*}
            C^{0,1}(M,\rho_{(a,b)}) & = C^{\emptyset,1}(M,\rho_{(a,b)})= \Zb^2;                                                              \\
            C^{1,1}(M,\rho_{(a,b)}) & = C^{\{ 1\},1}(M,\rho_{(a,b)})\oplus C^{\{2 \},1}(M,\rho_{(a,b)})                                                \\
                                    & = \left(\Zb/a\Zb \oplus\Zb\right)\oplus\left(\Zb \oplus\Zb/b\Zb\right); \\
            C^{2,1}(M,\rho_{(a,b)}) & = C^{E,1}(M,\rho_{(a,b)})=0.
        \end{align*}
        Because $d_{M,\rho_{(a,b)}}^{0,1}(x,y)=(x+a\Zb,y,x,y+b\Zb)$, we obtain
        \[
            H^{0,1}(M,\rho_{(a,b)}) = \kerne d_{M,\rho_{(a,b)}}^{0,1}=0
        \]
        and
        \begin{equation*}
            H^{1,1}(M,\rho_{(a,b)})= \frac{\Zb/a\Zb \oplus\Zb \oplus\Zb \oplus\Zb/b\Zb}{\im d_{M,\rho_{(a,b)}}^{0,1}}\cong \Zb/a\Zb \oplus \Zb/b\Zb.
        \end{equation*}

        Next, we see the case $j=2$. 
        By the definition,
        \begin{align*}
            C^{0,2}(M,\rho_{(a,b)}) & = C^{\emptyset,2}(M,\rho_{(a,b)})= \bigwedge^2\Zb^2\cong \Zb,                                                                                                                       \\
            C^{1,2}(M,\rho_{(a,b)}) & = C^{\{ 1\},2}(M,\rho_{(a,b)})\oplus C^{\{2 \},2}(M,\rho_{(a,b)})                                                                                                                                      \\
                                    & = \left(\bigwedge^2(\Zb/a\Zb \oplus\Zb)\right)\oplus\left(\bigwedge^2(\Zb \oplus\Zb/b\Zb)\right)\cong\Zb/a\Zb \oplus \Zb/b\Zb, \\
            C^{2,2}(M,\rho_{(a,b)}) & = C^{E,2}(M,\rho_{(a,b)})= \bigwedge^20=0
        \end{align*}
        where we use the isomorphism
        \begin{align*}
            \bigwedge^2(\Zb/a\Zb \oplus\Zb) & = \frac{(\Zb/a\Zb \oplus\Zb)\otimes(\Zb/a\Zb \oplus\Zb)}{((x_1,y_1)\otimes (x_2,y_2)+(x_2,y_2)\otimes (x_1,y_1))} \\
                                                                  & \cong \frac{\Zb/a\Zb\otimes\Zb \oplus\Zb\otimes\Zb/a\Zb}{(x_1\otimes y_2+y_1\otimes x_2)}                                \\
                                                                  & \cong \frac{\Zb/a\Zb \oplus\Zb/a\Zb}{(x_1+x_2)}\cong \Zb/a\Zb
        \end{align*}
        and similarly $\bigwedge^2(\Zb \oplus\Zb/b\Zb)\cong \Zb/b\Zb$. 
        Then $d_{M,\rho_{(a,b)}}^{0,2}(x)=(x+a\Zb,x+b\Zb)$. 
        Hence we obtain $H^{0,2}(M,\rho_{(a,b)}) = \kerne d_{M,\rho_{(a,b)}}^{0,2}\cong a\Zb\cap b\Zb\cong \Zb$ and $H^{1,2}(M,\rho_{(a,b)})\cong\Zb/\gcd(a,b)\Zb$.
    \end{proof}
    \section{Relations to other cohomology theories}\label{sec:relation}
    \subsection{Relation to the characteristic cohomology of hyperplane arrangements}
    In this subsection, we show that the characteristic cohomology of matroids is isomorphic to the characteristic cohomology of hyperplane arrangements given by Dancso and Licata \cite{DL}.
    Let $\Bbbk$ be a field of characteristic zero and let $V$ be a finite-dimensional $\Bbbk$-vector space endowed with a non-degenerate inner product.
    Let $\nu_e\ (e\in E)$ be vectors in $V$ and let $\mathcal{A} = \{ H_e\mid e\in E\}$ be the associated essential hyperplane arrangement, where $H_e= \nu_e^\perp$.
    For $S\subset [n]$, set $H_S= \bigcap_{e \in S} H_e$.
    Let $M(\mathcal A)$ be the matroid associated with $\mathcal{A}$.
    Define a $\Bbbk$-quasi-representation $\rho_\mathcal{A}$ of $M(\mathcal{A})$ by
    \begin{equation*}
        \rho_\mathcal{A}(S) = \left< \nu_e \mid e \in S \right>_\Bbbk \subset V.
    \end{equation*}
    For $S \subset [n]$ and $e \notin S$, let $p_{S,e}: H_S\to H_{S\cup e}$ be the orthogonal projection.
    This map induces an algebra map
    \begin{equation*}
        d_{arr}^{S,e}: \bigwedge^\bullet H_S \to \bigwedge^\bullet H_{S\cup e}.
    \end{equation*}
    For each $S\subset [n]$, we have $H_S=(\rho_\mathcal{A}(S))^\perp$.
    By the inner product, we have a natural isomorphism $H_S \cong V/\rho_\mathcal{A}(S)$.
    Moreover, under this identification, the orthogonal projection
    $p_{S,e}: H_S \to H_{S\cup e}$ corresponds to the natural projection $V/\rho_\mathcal{A}(S) \to V/\rho_\mathcal{A}(S\cup e)$.

    By taking exterior powers, we obtain a $\Bbbk$-module isomorphism
    \begin{equation*}
        \varphi_S: C_{arr}^{S,j}(\mathcal A,\Bbbk) = \bigwedge^j H_S \longrightarrow \bigwedge^j(V/\rho_\mathcal{A}(S)) = C^{S,j}(M(\mathcal A),\rho_\mathcal{A}).
    \end{equation*}
    
    Since the orthogonal projection $p_{S,e}: H_S\to H_{S\cup e}$ corresponds to the natural projection $V/\rho_\mathcal{A}(S) \to V/\rho_\mathcal{A}(S\cup e)$,
    the following diagram commutes:
    \begin{equation*}
        \xymatrix{
        C_{arr}^{S,j}(\mathcal A,\Bbbk) \ar[r]^{d_{arr}^{S,e}} \ar[d]_{\varphi_S} & C_{arr}^{S\cup e,j}(\mathcal A,\Bbbk) \ar[d]^{\varphi_{S\cup e}} \\
        C^{S,j}(M(\mathcal A),\rho_\mathcal{A}) \ar[r]^{d^{S,e}} & C^{S\cup e,j}(M(\mathcal A),\rho_\mathcal{A}).
        }
    \end{equation*}
    Therefore the chain complexes
    $\mathcal C_{arr}(\mathcal A,\Bbbk)$ and $\mathcal C(M(\mathcal A),\rho_\mathcal{A})$ are isomorphic.
    
    \begin{thm}\label{thm:cha_arr}
    Let $H_{arr}^{\bullet}(\mathcal{A}, \Bbbk)$ be the characteristic cohomology of an essential arrangement $\mathcal A$ over $\Bbbk$, and let $H^{\bullet}(M(\mathcal{A}),\rho_\mathcal{A})$ be the characteristic cohomology of the associated matroid with the above $\Bbbk$-quasi-representation. 
    Then
    \begin{equation*}
        H_{arr}^{\bullet}(\mathcal{A}, \Bbbk) \simeq H^{\bullet}(M(\mathcal{A}),\rho_\mathcal{A}).
    \end{equation*}
    \end{thm}

    \subsection{Relation to chromatic cohomology}
    In this subsection, we connect the characteristic cohomology of matroids to the chromatic cohomology given by Helme--Guizon and Rong \cite{HGR} via the characteristic cohomology of essentialized graphic arrangement.

    Let us review the construction of the chromatic cohomology. 
    Let $G = (V(G), E(G))$ be a graph with an order on $E(G)$. 
    We write $\mathrm{Conn}(S)$ for the set of connected components of the spanning graph $[G : S]$ for all subset $S \subset E(G)$.

    Let $c:\mathrm{Conn}(S) \to \{ 1,x\}\subset\Zb[x]/(x^2)$.
    For the labeled spanning graph $(S, c)$, called the \textbf{enhanced state}, define
    \begin{equation*}
        i(S) = \abs{S}, \ \text{and}\ j(S) = \big|\{ F \in \mathrm{Conn}(S) \mid c(F) = x\}\big|.
    \end{equation*}
    Then, define $C^{i,j}(G) = \left< (S,c) \mid S\subset E(G), i(S) = i, j(S) = j \right>$.
    This gives a chain group
    \begin{equation*}
        C^i(G) = \bigoplus_{j \geq 0}C^{i,j}(G)
    \end{equation*}
    graded by $j$. For simplicity, we will identify $[G:S]$ with an edge set $S$. 
    We define the differential
    \begin{equation*}
        d_{G}^{i,j}: C^{i,j}(G) \to C^{i+1,j}(G)
    \end{equation*}
    as
    \begin{equation*}
        d_{G}^{i,j}((S,c)) = \sum_{e \in E(G) \setminus S}\epsilon^{S,e}(S \cup e, c_{e})
    \end{equation*}
    where $n(S,e) = \abs{\{ e' \in S \mid e' < e \}}$, and $(S \cup e, c_{e})$ is defined as follows. 
    If $e$ is a bridge connecting two connected components $F_1$ and $F_2$ of $[G : S]$, then we define it by $c_{e}(F_1\cup F_2\cup e) = c(F_1) c(F_2)$, and $c_{e}(F_{3}) = c(F_{3})$, $F_{3}\neq F_1,F_2$.
    Note that the multiplication of labels is taken in the ring $\mathbb Z[x]/(x^2)$.
    In particular, $x^2=0$, and hence if both connected components are labeled by $x$, then the resulting enhanced state vanishes, i.e. $(S\cup e,c_e)=0$ in $C^{i+1,j}(G)$.
    If $e$ connects $F_1$ itself, then we define it by $c_{e}(F_1\cup e) = c(F_1)$, and $c_{e}(F_{3}) = c(F_{3})$, for all $F_{3}\neq F_1$.
    The map defined above gives a differential
    \begin{equation*}
        d_{G}^i= \bigoplus_{j \geq 0}d_{G}^{i,j}.
    \end{equation*}
    Thus we obtain a chain complex $\mathcal{C}(G) = (C^{\bullet}(G), d^{\bullet})$. 
    The resulting cohomology $H^{\bullet}(G)$ is the chromatic cohomology of graphs.
    \begin{thm}[\cite{HGR} Theorem 2.11] 
    The cohomology $H^{\bullet}(G)$ has a graded Euler characteristic equal to the chromatic polynomial evaluated by $1 + q$:
        \begin{equation*}
            \chi_{q}(H^{\bullet}(G)) = \chi(G, 1 + q).
        \end{equation*}
    \end{thm}

    Note that by Theorem 2.12 in \cite{HGR} the cohomology $H^{\bullet}(G)$ does not depend on the order of $E(G)$. 
    Without loss of generality fix a total order on the vertex set $V(G)= \{v_1<\dots<v_{\abs{V}}\}$. 
    Also fix a total order on the edge set $E(G)$ satisfying the following condition: 
    for two edges $e, e' \in E(G)$ with end points $u$, $v$ and $u', v'$ respectively, define $e<e'$ if $u<u'$, or $u=u'$ and $v<v'$.

    Let $\tilde{N}= \bigoplus_{v\in V(G)}\Zb v$ and let $I_{G}= \left< \min(V(F))\mid F\in \mathrm{Conn}(E(G))\right>\subset \tilde{N}$. 
    Note that $v_1\in I_{G}$.
    Define a $\Zb$-quasi-representation of a graphic matroid $M(G)$ by
    \begin{align*}
        \tilde\rho(S) & = \left< v-v'\mid v,v'\mbox{are the endpoints of some edge in }S\right>\subset\tilde{N}, \\
        \rho(S)       & = (\tilde\rho(S)+I_{G})/I_{G},                                                            \\
        N             & = \rho(E(G)).
    \end{align*}
    Then $\rho$ is a $\Zb$-canonical quasi-representation of $M(G)$. 
    In the rest of this subsection, we will omit the canonical quasi-representation $\rho$ and write $M(G)$ as $(M(G),\rho)$.

    By definition, $\tilde{N}/\tilde\rho(S)$ is the quotient module spanned by the set of vertices of $G$ by the relation in which a path exists between the vertices. 
    Thus $\tilde{N}/\tilde\rho(S)$ is isomorphic to the free module spanned by the set $\mathrm{Conn}(S)$ of connected components of $[G:S]$. 
    For $F\in \mathrm{Conn}(S)$, let $v_{S,F}$ denote the minimal vertex in $V(F)$. 
    We define a $\Zb$-module homomorphism as follows:
    \begin{equation*}
        \begin{aligned}
            \eta_{0}:C^{i,j}(G) \to & \bigoplus_{S\in\binom{E(G)}{i}}\bigwedge^{j}\left(\tilde{N}/\tilde{\rho}(S)\right) \\
            (S,c)\mapsto           & v_{S,c^{-1}(x)}= v_{S,F_1}\wedge\dots\wedge v_{S,F_j}
        \end{aligned}
    \end{equation*}
    where $\{F_1,\dots,F_{j}\}= \{F\in \mathrm{Conn}(S)\mid c(F)=x\}$ and $\min V(F_1)<\dots<\min V(F_{j})$. 
    Then the above homomorphism $\eta_{0}$ is an isomorphism.

    By definition, we have
    \begin{equation*}
        \frac{N}{\rho(S)}= \frac{(\tilde{\rho}(E(G))+ I_{G})/I_{G}}{(\tilde{\rho}(S)+ I_{G})/I_{G}}\cong\frac{\tilde{N}}{\tilde{\rho}(S)+ I_{G}}
    \end{equation*}
    where we use $\tilde{N}= \tilde{\rho}(E(G))+ I_{G}$ for the last isomorphism. 
    Thus we have the following isomorphism:
    \begin{equation*}
        \eta_1:\bigoplus_{S\in\binom{E(G)}{i}}\bigwedge^{j}\frac{\tilde{N}}{\tilde{\rho}(S)+I_{G}}\xrightarrow{\simeq}\bigoplus_{S\in\binom{E(G)}{i}}\bigwedge^{j}\frac{N}{\rho(S)}=C^{i,j}(M(G)).
    \end{equation*}

    Let $\pi:\bigoplus_{S\in\binom{E(G)}{i}}\bigwedge^{j}\left(\tilde{N}/\tilde{\rho}(S)\right) \to \bigoplus_{S\in\binom{E(G)}{i}}\bigwedge^{j}\left(\tilde{N}/(\tilde{\rho}(S)+I_{G})\right)$ be the wedge of the natural projections and let $\theta$ be the composition $\eta_1\circ\pi\circ\eta_{0}:C^{i,j}(G) \to C^{i,j}(M(G))$.
    \begin{lem}
        The map $\theta^{i,j}:C^{i,j}(G) \to C^{i,j}(M(G))$ is a chain map.
    \end{lem}
    \begin{proof}
        Let us show $\theta^{i+1,j}\circ d^{i,j}_{G}=d^{i,j}_{M(G)}\circ\theta^{i,j}$. 
        Let $(S,c)$ be an enhanced state. If $F,F'\in c^{-1}(x)$ and $e\in E$ is an edge connecting $F$ and $F'$, then we have $c_{e}(F\cup F'\cup e)=0$ and $(v_{F}+I_{G}+\tilde{\rho}(S\cup e))\wedge(v_{F'}+I_{G}+\tilde{\rho}(S\cup e))=0$. 
        Therefore it is sufficient to consider edges which do not connect components colored by $x$. 
        Then we obtain
        \begin{align*}
            \theta\circ d^{i,j}_{G}((S,c)) & = \theta\left(\sum_{e\in E(G)\setminus S}\epsilon^{S,e}(S\cup e,c_{e})\right)                                                               \\
                                           & = \sum_{e\in E(G)\setminus S}\epsilon^{S,e}\eta_1\circ\pi\circ\eta_{0}(S\cup e,c_{e})                                                     \\
                                           & = \sum_{e\in E(G)\setminus S}\epsilon^{S,e}\eta_1\circ\pi(v_{S\cup e,c_e^{-1}(x)})                                                        \\
                                           & = \sum_{e\in E(G)\setminus S}\epsilon^{S,e}(v_{S,F_1}+I_{G}+\tilde{\rho}(S\cup e))\wedge\dots\wedge (v_{S,F_j}+I_{G}+\tilde{\rho}(S\cup e)) \\
                                           & = d^{i,j}_{M(G)}\left((v_{S,F_1}+I_{G}+\tilde{\rho}(S))\wedge\dots\wedge (v_{S,F_j}+I_{G}+\tilde{\rho}(S))\right)                           \\
                                           & = d^{i,j}_{M(G)}(\eta_1\circ\pi(v_{S\cup e,c_e^{-1}(x)}))                                                                                 \\
                                           & = d^{i,j}_{M(G)}\circ \theta((S,c)).\qedhere
        \end{align*}
    \end{proof}

    \begin{lem}\label{thm: Homomorph} 
        There exists a homomorphism $\theta^\ast: H^\ast(G) \to H^\ast(M(G))$ induced by $\theta$.
    \end{lem}

    Furthermore, if the graph $G$ is connected, we are able to obtain a short exact sequence. 
    Let $\tau^{i,j-1}:C^{i,j-1}(M(G)) \to C^{i,j}(G)$ be a map defined by
    \begin{align*}
         & \tau^{i,j-1}\left((v_{S,F_1}+I_{G}+\tilde{\rho}(S))\wedge\dots\wedge (v_{S,F_{j-1}}+I_{G}+\tilde{\rho}(S))\right)                 \\
         & = \left\{\begin{aligned}(S,c^{\tau})&&\min(\mathrm{Conn}(S))\notin\{F_1,\dots,F_{j-1}\}\\ 0&&\mbox{otherwise}\\\end{aligned}\right.
    \end{align*}
    where $\min(\mathrm{Conn}(S))$ denotes the connected component containing the smallest vertex among all connected components of $[G:S]$, and $c^{\tau}(F)=x$ if and only if $F\in\{F_1,\dots,F_{j-1}\}$ or $F= \min(\mathrm{Conn}(S))$. 
    Then $\tau$ is a chain map because
    \begin{align*}
         & \tau^{i,j-1}\circ d^{i,j-1}\left((v_{S,F_1}+I_{G}+\tilde{\rho}(S))\wedge\dots\wedge (v_{S,F_{j-1}}+I_{G}+\tilde{\rho}(S))\right)                                                 \\
         & = \tau^{i,j-1}\left(\sum_{e\in E(G)\setminus E(S)}\epsilon^{S,e}\left((v_{S,F_1}+I_{G}+\tilde{\rho}(S\cup e))\wedge\dots\wedge (v_{S,F_{j-1}}+I_{G}+\tilde{\rho}(S\cup e))\right)\right) \\
         & = \sum_{e\in E(G)\setminus S}\epsilon^{S,e}(S\cup e,(c_{e})^{\tau}) = \sum_{e\in E(G)\setminus S}\epsilon^{S,e}(S\cup e,(c^{\tau})_{e})                                    \\
         & = d_{G}^{i,j}(S,c^{\tau})=d_{G}^{i,j}\circ\tau^{i,j}(S,c)
    \end{align*}
    for $\min(\mathrm{Conn}(S))\notin\{F_1,\dots,F_{j-1}\}$.

    \begin{lem}
        Let $G$ be a connected graph. 
        Then there exists a short exact sequence
        \begin{equation*}
            0\to C^{i,j-1}(M(G)) \xrightarrow{\tau^{i,j-1}}C^{i,j}(G)\xrightarrow{\theta}C^{i,j}(M(G)) \to 0.
        \end{equation*}
    \end{lem}
    \begin{proof}
        Let $\{F_1,\dots,F_{j-1}\}\subset\mathrm{Conn}(S)$ with $\min(V(F_1))<\dots<\min(V(F_{j-1}))$. 
        Then, by definition, we have
        \begin{align*}
             & \theta\circ\tau^{i,j-1}\left((v_{S,F_1}+I_{G}+\tilde{\rho}(S))\wedge\dots\wedge (v_{S,F_{j-1}}+I_{G}+\tilde{\rho}(S))\right)                          \\
             & = (v_{S,\min(\mathrm{Conn}(S))}+I_{G}+\tilde{\rho}(S))\wedge (v_{S,F_1}+I_{G}+\tilde{\rho}(S))\wedge\dots\wedge (v_{S,F_{j-1}}+I_{G}+\tilde{\rho}(S)) \\
             & = (I_{G}+\tilde{\rho}(S))\wedge (v_{S,F_1}+I_{G}+\tilde{\rho}(S))\wedge\dots\wedge (v_{S,F_{j-1}}+I_{G}+\tilde{\rho}(S))                              \\
             & = 0
        \end{align*}
        where we use $v_{S,\min(\mathrm{Conn}(S))}\in I_{G}$ in the third line. 
        Then, we need to show $\kerne \theta^{i,j}\subset\im \tau$. 
        Suppose $\sum_{(S,c)}a_{(S,c)}(S,c)\in \kerne \theta^{i,j}$. 
        Since $G$ is connected, $\kerne (\tilde{N}/\tilde{\rho}(S) \to N/\rho(S))$ is generated by $v_{S,\min(\mathrm{Conn}(S))}$. 
        Hence, we have that $\kerne \theta^{i,j}$ is generated by labeled spanning graphs $(S,c)$ such that $c(\min(\mathrm{Conn}(S)))=x$. 
        On the other hand, an enhanced state $(S,c)$ satisfying such conditions is an image of $(v_{S,F_1}+I_{G}+\tilde{\rho}(S))\wedge\dots\wedge (v_{S,F_{j-1}}+I_{G}+\tilde{\rho}(S))$ by $\tau$ where $\{F_{\min(\mathrm{Conn}(S))},F_1,\dots,F_{j-1}\} =c^{-1}(x)$.
    \end{proof}

    \begin{thm}\label{cor:long_ex} 
        Let $G$ be a connected graph. 
        Then there is a long exact sequence
        \begin{equation*}
            \cdots\to H^{i,j-1}(M(G)) \to H^{i,j}(G) \to H^{i,j}(M(G)) \to H^{i+1,j-1}(M(G)) \to \cdots.
        \end{equation*}
    \end{thm}
    \begin{rem}
        Similar to the above long exact sequence, we can also obtain a long exact sequence
        \begin{equation*}
            \cdots\to H^{i,j-1}_{arr}(\mathcal{A}_{ess}) \to H^{i,j}_{arr}(\mathcal{A}) \to H^{i,j}_{arr}(\mathcal{A}_{ess}) \to H^{i+1,j-1}_{arr}(\mathcal{A}_{ess}) \to \cdots
        \end{equation*}
        for a hyperplane arrangement $\mathcal{A}$ with a one-dimensional center and its essentialization $\mathcal{A}_{ess}= \{H/\bigcap_{H\in \mathcal{A}}H\mid H\in \mathcal{A}\}$.
    \end{rem}

    For every graphic matroid $M(G)$, there exists some connected graph $G'$ such that $M(G)\cong M(G')$. 
    By combining the above with Corollary 13 in \cite{HGPR}, we obtain the following corollary.
    \begin{cor}
        Let $G$ be a connected graph with $n$ vertices. 
        If $i + j < n - 1$, then
        \begin{equation*}
            H^{i,j}(M(G))\cong H^{i+1,j-1}(M(G)).
        \end{equation*}
    \end{cor}

\bibliographystyle{amsalpha}
\bibliography{matroid_hlg}

\end{document}